\documentclass{amsart}
\usepackage{graphicx,amsthm,amsmath,amsfonts,amssymb,epsfig}
\newtheorem{theorem}{Theorem}[section]
\theoremstyle{plain}

\newtheorem{corollary}{Corollary}[section]

\newtheorem{lemma}{Lemma}[section]

\newtheorem{proposition}{Proposition}[section]
\newtheorem{remark}{Remark}[section]

\numberwithin{equation}{section}

\begin{document}
\title[On Seifert conemanifold structures]{On continuous families of geometric Seifert conemanifold structures }
\author{Mar\'{\i}a Teresa Lozano}
\address[M.T. Lozano]{IUMA, Universidad de Zaragoza \\
Zaragoza, 50009, Spain}
\email[M.T. Lozano]{tlozano@unizar.es}
\thanks{Partially supported by grant MTM2013-45710-c2-1-p and Grupo consolidado E15 Gobierno de Arag\'{o}n/Fondo Social Europeo.}
\author{Jos\'{e} Mar\'{\i}a Montesinos-Amilibia}
\address[J.M. Montesinos]{Dto. Geometr\'{\i}a y Topolog\'{\i}a \\
Universidad Complutense, Madrid 28080 Spain}
\email[J.M. Montesinos]{jose{\_}montesinos@mat.ucm.es}
\thanks{Partially supported by grant MTM2015-63612-P.}
\date{\today }

\date{\today }
\subjclass[2000]{Primary 20G20, 53C20; Secondary 57M50, 22E99}
\keywords{hyperbolic manifold, ,}
\dedicatory{Dedicated to Louis H. Kauffman on the Occasion of his 70th Birthday}

\begin{abstract}
We determine the  Thurston's geometry possesed by any Seifert fibered conemanifold structure in a Seifert manifold with orbit space $S^2$ and no more than three exceptional fibres, whose singular set, composed by fibres, has at most 3 components which can include exceptional or general fibres (the total number of exceptional and singular fibres is less or equal than three). We also give the method to obtain the holonomy of that structure. We apply these results to three families of Seifert manifolds, namely,  spherical, Nil manifolds  and  manifolds obtained by Dehn surgery in a torus knot $K_{(r,s)}$.  As a consequence we generalize to all torus knots the results obtained  in \cite{LM2015}  for the case of the left handle trefoil knot. We associate a plot to each torus knot for the different geometries, in the spirit of Thurston.

\end{abstract}

\maketitle
\tableofcontents

\section*{\protect\bigskip Introduction}

An \emph{orbifold} is a topological space which locally is as the quotient of a ball by a properly discontinuous group. In dimension 3 a \emph{geometric orbifold} is a quotient $\Gamma \backslash X$ where $X$ is one of the eight Thurston geometries and the action of $\Gamma <$ Isom$(X)$ on $X$ is properly discontinuous. The \emph{singularity} is the projection of the set of fixed points of the   elements of $\Gamma$. A general description of these geometric orbifolds is given in \cite{D1988}.  A \emph{Seifert fibered orbifold} is an orbifold structure in a Seifert manifold where the singular set is a finite number of fibres. Each fibre $L_{i}$ in the singular set is endowed with a natural number $n_{i}$, whose meaning is that the angle around the fibre is $2\pi /n_{i}$.

The concept of \emph{Seifert fibered conemanifold} generalices that of Seifert fibered orbifold, by allowing any values for the angles around the singular fibres. In this way the Seifert orbifolds with a fixed singular set are included in a continuous family of Seifert fibered conemanifolds. This continuous family of conemanifolds is the convenient framework to study degenerations and transitions among different geometries, for instance Spherical-Euclidean-Hyperbolic (\cite{HLM1992},\cite{HLM1995}), $(S^2\times \mathbb{R})$-$E^3$-$(H^2\times\mathbb{R})$ or Spherical-Nil-$\widetilde{SL(2,\mathbb{R})}$ (\cite{LM2014}). The case of the  Seifert fibered conemanifold structures on the manifolds obtained by Dehn surgery in the  left-handed trefoil knot were studied in \cite{LM2015}. Here we study the case of Seifert fibered conemanifold structures on the Seifert manifolds with orbit space $S^2$ and with no incompressible fiberwise torus such that the singular set is a link with no more than three components which can include exceptional or general fibres (the total number of exceptional and singular fibres is less or equal than three). This family includes some interesting subfamilies, as the Seifert manifolds with orbit space $S^2$ and finite fundamental group, and also the Seifert fibered conemanifold structures in manifolds obtained by Dehn surgery in a torus knot $K_{(r,s)}$ with singularity the core of the surgery. We obtain for each torus knot its two limits of sphericity, to be explained later.   As a  consequence,  we can obtain the holonomy of the Thurston geometry possessed by any given Seifert fibered orbifold obtained by surgery on a torus knots. The method to do this is explained in this paper and a concrete example of the application of this method was developed in \cite{LM2015} for the case of the left handle trefoil knot. In other words, among other results, in this paper we are generalizing to all torus knot the results obtained  in \cite{LM2015}. We associate a plot for the different geometries, similar to Thurston's.

The paper is  organized as follows. In Section \ref{s1} we describe the family $\mathcal{F}$ of Seifert manifolds to study. In Section \ref{s2} the geometries of 2-conemanifold structures in $S^2$ with 2 or 3 cone points are obtained. Those structures are the ones on the orbit spaces of the manifolds in $\mathcal{F}$. Section \ref{s3} contains the theorems that describe the geometric conemanifold structures on the manifolds in $\mathcal{F}$ for the cases of 3, 2 and 1 singular fibres. In the following sections we apply these theorems to the subfamilies of spherical manifolds (Section \ref{s4}), Nil manifolds (Section \ref{s5}) and  manifolds obtained by Dehn surgery in a torus knot (Section \ref{s5}).

We use the Seifert notation \cite{S1933}, and the currently standard used  orientation convention for lens spaces and Dehn surgery (see \cite{M1987}), though it may be the opposite of the ones used in \cite{M1971} and \cite{O1972}.

\section{Some Seifert manifolds}\label{s1}

We consider the family $\mathcal{F}$ of closed Seifert manifolds $M$ whose orbit space is the 2-sphere such that the number of exceptional fibres is less or equal than 3. The following  Seifert signature defines the oriented Seifert manifold up to fibre and orientation preserving homeomorphism.
\[
M=\left\{
        \begin{array}{l}
          (O,o,0\, |\, b;(a_{1},b_{1}),(a_{2},b_{2}),(a_{3},b_{3})) \\
         (O,o,0\, |\, b;(a_{1},a_{1}),(a_{2},b_{2}))\\
         (O,o,0\, |\, b;(a_{1},a_{1}))\\
         (O,o,0\, |\, b))
        \end{array}
      \right.
\]
where $O$ stands for orientable 3-manifold; $o$ stands for orientable orbit space; $0$, for the genus of the orbit space ($S^2$ in our cases); $b$ is an integer; $(a_{i},b_{i})$ are coprime integers such that $0<b_{i}<a_{i}$.

If the Euler number
\[
e= -b-\underset{i}{\sum }\frac{b_i}{a_i}
\]
is different from zero, the manifold $M$ have a spherical, Nil or $\widetilde{SL(2,\mathbb{R})}$ geometric structure. If the Euler number is zero, the possible geometries are $S^2\times \mathbb{R}$, Euclidean or $H^2\times \mathbb{R}$.  Which one of these  geometries is the correct one depends on the value of the orbifold Euler characteristic of the orbit space $\chi (\mathbf{B})$(See for instance, \cite{S1983}).
\medskip
\begin{center}
 \begin{tabular}{c||c|c|c|}
      & $\chi (\mathbf{B})>0$ & $\chi (\mathbf{B})=0$ & $\chi (\mathbf{B})<0 $\\ \hline \hline
      &&&\\
     $e=0$ & $S^2\times \mathbb{R}$ & $E^3$ & $H^2\times \mathbb{R}$ \\ \hline &&&\\
     $e\neq 0$ & $S^3$ & Nil & $\widetilde{SL(2,\mathbb{R})}$ \\ \hline
   \end{tabular}
\end{center}
\medskip

Observe that each manifold in this family $\mathcal{F}$ is \textit{atoroidal} or \emph{small}, that is, contain no incompressible fiberwise tori. In fact if a Seifert manifold is atoroidal then either it belongs to  $\mathcal{F}$ or its orbit base is the projective plane with at most one cone point \cite{FM1997}.

This family includes some interesting subfamilies:
\begin{description}
  \item[$\mathcal{F}_1$] Seifert manifolds with orbit space $S^{2}$ and spherical geometry. They have finite fundamental group. It includes all the fibrations in $S^{3}$ with two exceptional fibres, with multiplicities  a pair of coprime integers, and whose general fibre is a torus knot.
  \item[$\mathcal{F}_2$] Seifert manifolds with orbit space $S^{2}$ and Nil or Euclidean geometry.
\item[$\mathcal{F}_3$] Manifolds obtained by  $p/q$-Dehn surgery on a torus knot in $S^{3}$.
  \item[$\mathcal{F}_4$] Homological spheres with 3 exceptional fibres. They are the Brieskorn complete intersection manifolds $M(a_1,a_2,a_3)$, where $\gcd (a_{i},a_{j})=1$, $0<i<j\leq 3$.
\end{description}

These subfamilies have intersections. In fact, the Poincar\'{e} manifold
\[
(O,o,0\, |\, -1;(2,1),(3,1),(5,1))
\]
belongs to $\mathcal{F}_1$, $\mathcal{F}_3$ and $\mathcal{F}_4$. More generally, the manifolds in $\mathcal{F}_4$ where $a_1<a_2<a_3$, $a_3=|qa_1 a_2 -1|$ are the result of $1/q$-Dehn surgery on the torus knot $K_{(a_{2},a_{1})}$, belonging to $\mathcal{F}_{3}$, \cite{S1933}.

Consider  a geometric Seifert fibered conemanifold structure on $M\in \mathcal{F}$ such that (i) the singular set $L$, if any, consists of fibres (singular or not); and (ii) the union of $L$ with the set of exceptional fibres  has less or equal to three components.

Then the geometry on $M$ is \emph{compatible} with the fibred Seifert structure. That is, the fibres are geodesics and the orbit space $\mathbf{B}$ of $M^{3}$ inherits a 2-dimensional geometric cone manifold structure with spherical, Euclidean or hyperbolic geometry, as the case may be, and the singular set consists  at most of 3 cone-points.

To unify notation, in the case that the Seifert manifold have less than three exceptional fibres, we will add, at the right of the signature, appropriated symbols for general fibres. The signature, written between angle brackets, means that the signature is not normalized (see for instance page 145 in \cite{M1987}. Then, the manifold $M$, objet of our study, has the following signature
\[
M=\left\langle O,o,0\, |\, b;(a_{1},b_{1}),(a_{2},b_{2}),(a_{3},b_{3})\right\rangle ,
\]
where $(a_{i},b_{i})$ are coprime integers, $0<b_{i}<a_{i}$, when $a_i>1$; or $(a_{i},b_{i})=(1,0)$.  We also assume that $a_1\geq a_2\geq a_3$.

The change $\left\{b,(a_i,b_i)\right\}\longrightarrow \left\{b-r,(a_i,b_i+ra_i)\right\}$ in the signature does not change the manifold (\cite{M1987}). Using this, we will often unify the signature for $M$, that we will  use in the sequel, as follows
 \[
M=\left\langle O,o,0\, |\, 0;(a_{1},b_{1}),(a_{2},b_{2}),(a_{3},b_{3}+ba_3)\right\rangle .
\]

\section{The 2-conemanifold structure on the orbit space}\label{s2}

A notation for a 2-dimensional closed orientable orbifold is given by the signature
\[
\left( O ,\, g\, |\, a_1, ...,a_s\right)
\]
where $O$ stands for orientable surface; $g$, for the genus of the surface; and $a_1, ...,a_s$ are the natural numbers associated to the $s$ cone points $(P_1,...,P_s)$, meaning that the angle around the $P_i$ cone point is $2\pi /a_{i}$.

If the manifold $M\in \mathcal{F}$ has a  conemanifold structure with angles $\beta_i=2\alpha_i a_i$ around the $(a_i,b_i)$-fibres, the orbit space $\mathbf{B}$ has a cone manifold structure in the 2-sphere $S^{2}$ with three cone points with angles $(2\alpha_1, 2\alpha_2, 2\alpha_3 )$, where we will assume $0\leq \alpha_i\leq \pi$. By natural generalization of the above orbifold signature to a cone manifold structure, we denoted the induced 2-conemanifold structure on $\mathbf{B}$ by $(O,0\, |\, \pi /\alpha_1,\pi /\alpha_2,\pi /\alpha_3 )$.  The geometry on the 2-conemanifold
\[
(O,0\, |\, \pi /\alpha_1,\pi /\alpha_2,\pi /\alpha_3 )
\]
 depends on the value of  the angles $\alpha_i$.

\begin{proposition}
For $0< \alpha_i< \pi$, the 2-conemanifold $(O,0\, |\, \pi /\alpha_1,\pi /\alpha_2,\pi /\alpha_3 )$ supports spherical, Euclidean or hyperbolic geometry according as  $\alpha_1+\alpha_2+\alpha_3$ is $>,=,<$ than $\pi$. (\cite{T1980}).\qed
\end{proposition}

In order to study the continuous deformation of the geometric conemanifold structures  in the Seifert manifold $M$, we are interested in the description of the corresponding continuous family of geometric conemanifold structures in the orbit space $\mathbf{B}$ of the Seifert manifold. It is clear that the Euclidean case $\alpha_1+\alpha_2+\alpha_3=\pi$ is a limit case between the spherical and the hyperbolic case. We are interested in  all the possible limits of sphericity or hyperbolicity, for all possible values $0\leq \alpha_i\leq \pi$.

The  hyperbolic metric in the Poincar\'{e} disc model, the open unit disc
\[
D_{1}=\left\{ z=x+iy \,|\, x^{2}+y^{2}<1\right\}
\]
is given by
\[
ds^2=\frac{4dzd\overline{z}}{(1-z\overline{z})^{2}}.
\]

In order to study the degeneration on geometric structures, it is more convenient to work with a disc in $\mathbb{C}$ with radius $\frac{1}{\sqrt{S}}$. Then the dilatation
 \[
\begin{array}{cccc}
  \lambda :& D_{\frac{1}{\sqrt{S}}} & \longrightarrow & D_{1} \\
 & z & \to & \sqrt{S}z
\end{array}
\]
is an isometry if and only if the metric on $D_{\frac{1}{\sqrt{S}}}$ is given by
\[ ds^2=\frac{4Sdzd\overline{z}}{(1-Sz\overline{z})^{2}},\quad 1\geq S>0 \]

\begin{figure}[h]
\begin{center}\epsfig{file=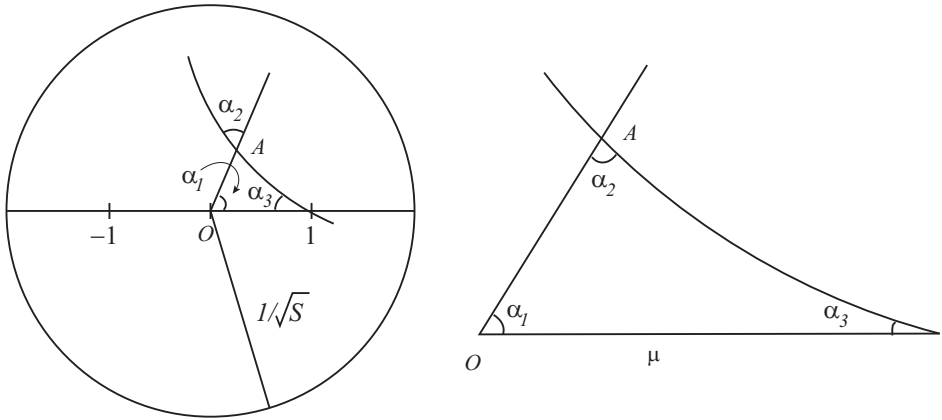,height=5.5cm}\caption{The hyperbolic plane and a hyperbolic triangle.}\label{fpqahyper}
\end{center}
\end{figure}

Let $T_{(\alpha_1,\alpha_2,\alpha_3)}$, $0<\alpha_1,\alpha_2,\alpha_3<\pi$, be a hyperbolic triangle. Figure \ref{fpqahyper} shows the triangle $T_{(\alpha_1,\alpha_2,\alpha_3)}$ with  the $\alpha_1$ and $\alpha_3$ angles in the vertices placed at the points $0$ and $1$ respectively and with the $\alpha_2$ angle at the remaining vertex. Let us relate the angles with the parameter $S$, using the trigonometric formulas for hyperbolic triangles.
\begin{equation*}
    \cosh \mu =\frac{\cos \alpha_2+\cos \alpha_3 \cos \alpha_1}{\sin \alpha_3 \sin  \alpha_1}
\end{equation*}

\begin{eqnarray*}
  \mu &=& \int_{0}^{1}\sqrt{\frac{4S}{(1-St^{2})^{2}}}dt=\int_{0}^{1}\frac{2\sqrt{S}}{(1-St^{2})}dt=\\
   &=& \lg (1+\sqrt{S}t)-\lg (1-\sqrt{S}t)|_{0}^{1}= \lg \frac{1+\sqrt{S}}{1-\sqrt{S}}
\end{eqnarray*}

\begin{equation*}
  \Rightarrow \quad  e^{\mu }=\frac{1+\sqrt{S}}{1-\sqrt{S}}\quad \Rightarrow\quad \cosh \mu =\frac{e^{2\mu} +1}{2e^{\mu }}=\frac{1+S}{1-S}
\end{equation*}
Therefore
\begin{equation*}
\frac{\cos \alpha_2+\cos \alpha_3 \cos \alpha_1}{\sin \alpha_3 \sin  \alpha_1}=\frac{1+S}{1-S}\quad \Rightarrow
S=\frac{\cos \alpha_2+\cos \alpha_1 \cos \alpha_3-\sin \alpha_1 \sin \alpha_3 }{\cos \alpha_2+\cos \alpha_1 \cos \alpha_3+\sin \alpha_1 \sin \alpha_3 }
\end{equation*}
\begin{equation}\label{esalfah}
\boxed{S=\frac{\cos \alpha_2+\cos (\alpha_1 + \alpha_3) }{\cos \alpha_2+\cos (\alpha_1 - \alpha_3) }}
\end{equation}

Analogously, $\mathbb{C}P^{1}$ with the spherical Riemannian metric,
 \[ ds^2=\frac{-4Sdzd\overline{z}}{(1-Sz\overline{z})^{2}},\quad S<0 \]
is the stereographic projection of the sphere $S^{2}$ with radius $\frac{1}{\sqrt{-S}}$ endowed with a Riemannian metric isometric to the usual spherical metric on the unit sphere in $\mathbb{R}^{3}$. The circle of radius $\frac{1}{\sqrt{-S}}$ is the equator. Figure \ref{fpqaspher} shows the spherical triangle $T_{(\alpha_1,\alpha_2,\alpha_3)}$ analogous to the hyperbolic case.
\begin{figure}[h]
\begin{center}\epsfig{file=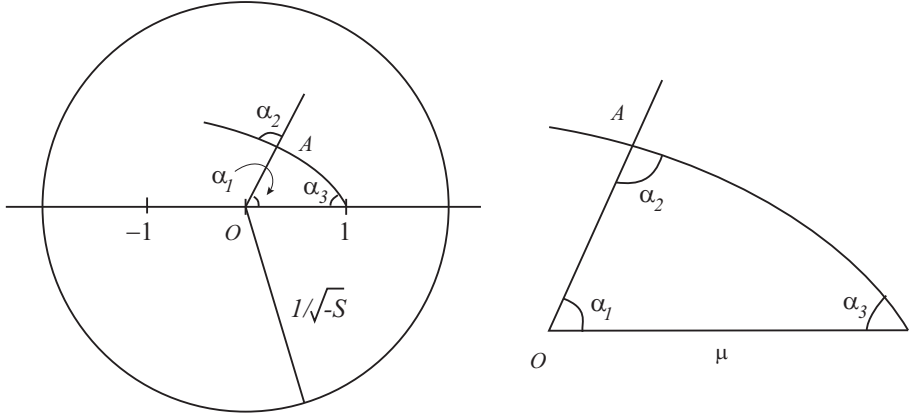,height=5.5cm}\caption{The spherical case.}\label{fpqaspher}
\end{center}
\end{figure}

In this spherical case, applying the cosine rule
\begin{equation*}
    \cos \mu =\frac{\cos \alpha_2+\cos {\alpha_3 }\cos \alpha_1}{\sin \alpha_3 \sin  \alpha_1}
\end{equation*}

\begin{eqnarray*}
  \mu = \int_{0}^{1}\sqrt{\frac{-4S}{(1-St^{2})^{2}}}dt=\int_{0}^{1}\frac{i 2\sqrt{S}}{(1-St^{2})}dt\\
   \Rightarrow \quad -i \mu= \lg (1+\sqrt{S}t)-\lg (1-\sqrt{S}t)|_{0}^{1}= \lg \frac{1+\sqrt{S}}{1-\sqrt{S}}
\end{eqnarray*}

\begin{equation*}
  \Rightarrow \quad  e^{-i\mu }=\frac{1+\sqrt{S}}{1-\sqrt{S}}\quad \Rightarrow\quad \cos \mu =\frac{e^{i\mu} +e^{-i\mu}}{2}=\frac{1+S}{1-S}
\end{equation*}
Therefore as before
\begin{equation*}
\frac{\cos \alpha_2+\cos {\alpha_3 }\cos \alpha_1}{\sin \alpha_3 \sin  \alpha_1}=\frac{1+S}{1-S}\quad \Rightarrow S=\frac{\cos \alpha_2+\cos \alpha_1 \cos \alpha_3-\sin \alpha_1 \sin \alpha_3 }{\cos \alpha_2+\cos \alpha_1 \cos \alpha_3+\sin \alpha_1 \sin \alpha_3 }
\end{equation*}
\begin{equation}\label{esalfas}
\boxed{S=\frac{\cos \alpha_2+\cos (\alpha_1 + \alpha_3) }{\cos \alpha_2+\cos (\alpha_1 - \alpha_3) }}
\end{equation}
Observe that this formula (\ref{esalfas}) for the spherical case, coincides with the formula (\ref{esalfah}) for the hyperbolic case.

The limit cases, when the geometry fails to be hyperbolic or spherical and  becomes Euclidean, are obtained when the radius is $\infty$, that is, when $S=0$. Next we compute the relationship between the values of the angles $\alpha_i$ making $S=0$.
\[
S=0 \quad \Leftrightarrow \quad \left\{ \begin{array}{l}
                                  N(S):= \cos \alpha_2+\cos (\alpha_1 + \alpha_3)=0;\,  \text{and} \\
                                   D(S):=\cos \alpha_2+\cos (\alpha_1 - \alpha_3)\neq 0
                                 \end{array}\right.
\]

Then,
\[
N(S)=0 \quad \Leftrightarrow \quad
\left\{ \begin{array}{l}
\cos \alpha_2=-\cos (\alpha_1 + \alpha_3)=\cos (\pi-\alpha_1 - \alpha_3);
 \quad\text{or}\\
\cos (\alpha_1 + \alpha_3)=-\cos \alpha_2=\cos (\pi -\alpha_2 ).
\end{array}\right.
\]

 This implies that
\begin{equation*}
 \alpha_2= \left\{\begin{array}{l}
\pi-\alpha_1-\alpha_3 \quad \Rightarrow\quad \alpha_1+\alpha_2+\alpha_3=\pi ;
 \quad\text{or}\\
2\pi-(\pi-\alpha_1-\alpha_3)=\pi+\alpha_1+\alpha_3 \quad \Rightarrow\quad -\alpha_1+\alpha_2-\alpha_3=\pi
\end{array}\right.
\end{equation*}
or
\begin{equation*}
 \alpha_1 + \alpha_3= \left\{\begin{array}{l}
\pi-\alpha_2\quad \Rightarrow\quad \alpha_1+\alpha_2+\alpha_3=\pi ;
 \quad\text{or}\\
2\pi-(\pi-\alpha_2)=\pi+\alpha_2 \quad \Rightarrow\quad \alpha_1-\alpha_2+\alpha_3=\pi
\end{array}\right.
\end{equation*}

On the other hand, if $D(S)=0$ when $N(S)=0$,
\[
0=\cos \alpha_2+\cos (\alpha_1 + \alpha_3)=\cos \alpha_2+\cos (\alpha_1 - \alpha_3) \quad \Rightarrow
\]
\[
\cos (\alpha_1 + \alpha_3)=\cos (\alpha_1 - \alpha_3)\quad \Rightarrow\quad \alpha_1 + \alpha_3=
\left\{ \begin{array}{l}
\alpha_1 -\alpha_3 \Rightarrow \alpha_3=0;
 \quad\text{or}\\
2\pi-(\alpha_1- \alpha_3)\Rightarrow \alpha_1=\pi
\end{array}\right\}
\]

For $\alpha_3=0$, $S=1$ and the triangle is hyperbolic. For $\alpha_1=\pi$, $S=1$ and the triangle is spherical. In this case, $\alpha_2=\alpha_3$. See Figure \ref{funopi}.

\begin{figure}[h]
\begin{center}\epsfig{file=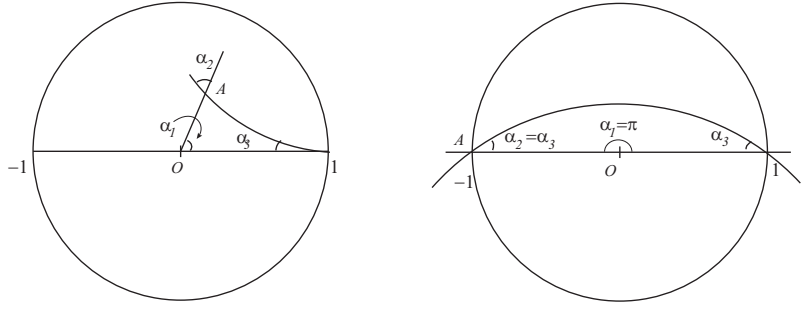,height=4cm}
\caption{The cases $\alpha_3=0$ and $\alpha_1=\pi$}\label{funopi}\end{center}
\end{figure}

Then $S=0$ for $\alpha_3\neq 0$, $\alpha_1\neq \pi$ and

\begin{equation}\label{solucionalfas}
\begin{array}{lr}
 p_1:\quad& \alpha_1+\alpha_2+\alpha_3=\pi ;\,\text{or} \\
 p_2:\quad& -\alpha_1+\alpha_2-\alpha_3=\pi ;\,\text{or} \\
p_3:\quad& \alpha_1-\alpha_2+\alpha_3=\pi .
\end{array}
\end{equation}

Consider the 3-dimensional space with coordinates $(\alpha_1,\alpha_2,\alpha_3)$. The points in this space defining the values for the angles of a possible triangle belong to the cube $[0,\pi]\times  [0,\pi]\times [0,\pi]$. Observe that if the point $(\alpha_1,\alpha_2,\alpha_3)$ defines an oriented triangle, the points $(\alpha_2,\alpha_3,\alpha_1)$ and $(\alpha_3,\alpha_1,\alpha_2)$ define the same oriented triangle up to orientation preserving isometry. Therefore all the angle conditions and equations of planes, that delimit regions containing points corresponding to triangles in the same geometry, must have a 3-cyclic symmetrical position around the diagonal line $(\alpha,\alpha,\alpha)$ with vertices $(0,0,0),\,(\pi ,\pi ,\pi)$. They are
\begin{equation}\label{solucionalfas2}
\left\{ \begin{array}{lr}
&\alpha_2\neq 0,\quad \alpha_3\neq \pi ; \,\text{and}\\
p_4: \quad& \alpha_1-\alpha_2-\alpha_3=\pi ;\,\text{or}\\
p_5: \quad& -\alpha_1+\alpha_2+\alpha_3=\pi .
\end{array}\right.
\end{equation}
\begin{equation}\label{solucionalfas3}
\left\{\begin{array}{lr}
&\alpha_1\neq 0,\quad \alpha_2\neq \pi ; \,\text{and}\\
p_6: \quad& -\alpha_1-\alpha_2+\alpha_3=\pi ;\,\text{or} \\
p_7: \quad& \alpha_1+\alpha_2-\alpha_3=\pi .
\end{array}\right.
\end{equation}

The first equation, $p_1$, in (\ref{solucionalfas}), $\alpha_1+\alpha_2+\alpha_3=\pi$, defines a plane which intersects the cube in the triangle with vertices $ (\pi,0,0),(0,0,\pi),(0 ,\pi ,0 )$. See Figure \ref{fregionesalfaEH}. Each point in the interior of this triangle defines a Euclidean triangle in the Euclidean space. The points in the cube at the side $\alpha_1+\alpha_2+\alpha_3<\pi$ define  hyperbolic triangles; and at side $\alpha_1+\alpha_2+\alpha_3>\pi$ the points may define a spherical triangle as we presently explain.

\begin{figure}[h]
\begin{center}\epsfig{file=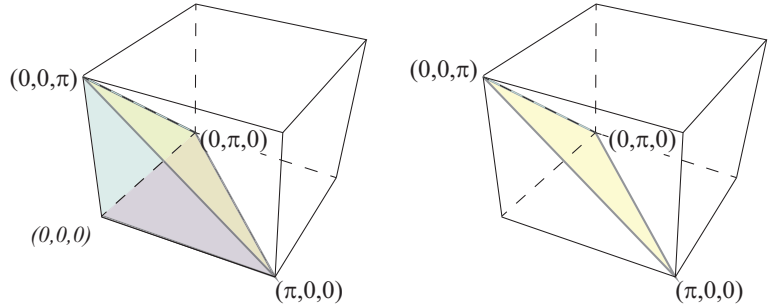,height=4cm}
\caption{The hyperbolic region and the Euclidean case}\label{fregionesalfaEH}\end{center}
\end{figure}

 The planes defined by equations  $p_2$, $p_4$ and $p_6$ in (\ref{solucionalfas}), (\ref{solucionalfas2}) and (\ref{solucionalfas3}), intersect the cube just in a vertex. They do not delimite any region in the cube.

The planes defined by the equations $p_3$, $p_5$ and $p_7$ intersect the cube in the triangles with vertices
\[
\begin{array}{ccc}
p_3:&\quad &(\pi,0,0),(0,0,\pi),(\pi ,\pi ,\pi )\\
p_5:&\quad &(0,\pi,0),(0,0,\pi),(\pi ,\pi ,\pi )\\
p_7:&\quad &(\pi,0,0),(0,\pi,0),(\pi ,\pi ,\pi )
\end{array}
\]

\begin{figure}[h]
\begin{center}\epsfig{file=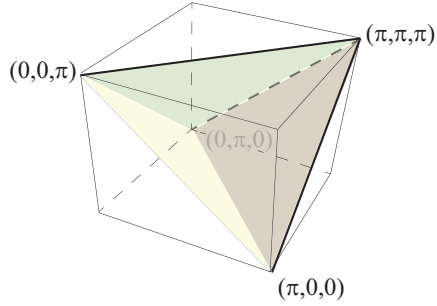,height=4cm}
\caption{The spherical region}\label{fregionesalfaS}\end{center}
\end{figure}

Points in the interior of the tetrahedron delimited by $p_1,\,p_3,\,p_5,\,p_7$, with vertices $(\pi,0,0),(0,\pi,0),(0,0,\pi),(\pi ,\pi ,\pi )$, represent spherical triangles. (See Figure \ref{fregionesalfaS}).

Points in the interior of the faces $p_3,\,p_5,\,p_7$ of this tetrahedon  correspond to the limit of  spherical triangles in Euclidean space ($S=0$). The limit is no longer a triangle (in  Euclidean space) because the sum of the three coordinates $\alpha_1$, $\alpha_2$ and $\alpha_3$ is bigger than $\pi$. The points lying on these faces represent \emph{upper limits of sphericity}. The face $p_1$ is the \emph{lower limit of sphericity}. (Compare \cite{LM2015}).

Points in the edges  $l_1=\left( (\pi,0,0),(\pi ,\pi ,\pi )\right)$, $l_2=\left( (0,\pi,0),(\pi ,\pi ,\pi )\right)$ and $l_3=\left( (0,0,\pi),(\pi ,\pi ,\pi )\right)$, correspond to spherical triangles $T_{(\pi ,\alpha ,\alpha )}$, $T_{(\alpha ,\pi ,\alpha )}$ and $T_{(\alpha,\alpha ,\pi )}$, respectively.

\begin{figure}[h]
\begin{center}\epsfig{file=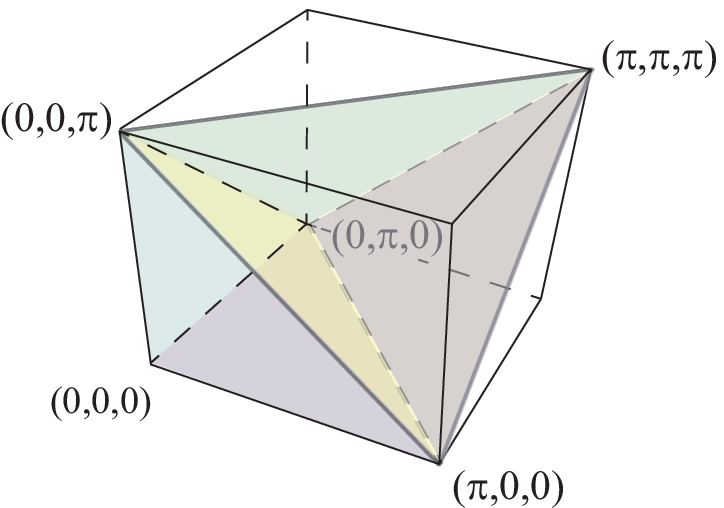,height=4cm}
\caption{}\label{fregionesalfa}\end{center}
\end{figure}

We collect in the following lemma the kind of geometry possessed by the elements of the continuous family $\left\{ T_{(\alpha_1,\alpha_2,\alpha_3)}\right\}$ of triangles.

\begin{lemma}
The  triangle  $T_{(\alpha_1,\alpha_2,\alpha_3)}$  is hyperbolic  for points in the closed tetrahedron
\[
t_h=\left( (0,0,0),(\pi,0,0),(0,\pi,0),(0,0,\pi)\right),
\]
excepting  the points in the face
\[f_e= \left((\pi,0,0),(0,\pi,0),(0,0,\pi)\right).\]

Points in the interior of the face $f_e$ correspond to  Euclidean triangles. This face is the limit of hyperbolicity and the lower limit of sphericity.

The triangle $T_{(\alpha_1,\alpha_2,\alpha_3)}$ is spherical for points in the interior of the tetrahedron  \[
t_s=\left((\pi,0,0),(0,\pi,0),(0,0,\pi),(\pi ,\pi ,\pi )\right).
 \]

 For points in the interior of the faces
\begin{eqnarray*}
 f_1&=&\left((0,\pi,0),(0,0,\pi),(\pi ,\pi ,\pi )\right), \\
f_2&=&\left((\pi,0,0),(0,0,\pi),(\pi ,\pi ,\pi )\right) \, \text{and} \\
 f_3&=&\left((\pi,0,0),(0,\pi,0),(\pi ,\pi ,\pi )\right)
\end{eqnarray*}
   the geometry is Euclidean but $T_{(\alpha_1,\alpha_2,\alpha_3)}$ is not a Euclidean triangle.

   Points in the edges  \begin{eqnarray*}
     l_1 &=& \left( (\pi,0,0),(\pi ,\pi ,\pi )\right), \\
     l_2 &=& \left( (0,\pi,0),(\pi ,\pi ,\pi )\right)\, \text{and} \\
     l_3 &=& \left( (0,0,\pi),(\pi ,\pi ,\pi )\right)
     \end{eqnarray*}
     correspond to spherical triangles. Figure \ref{fregionesalfa}. \qed
\end{lemma}

\begin{figure}[h]
\begin{center}\epsfig{file=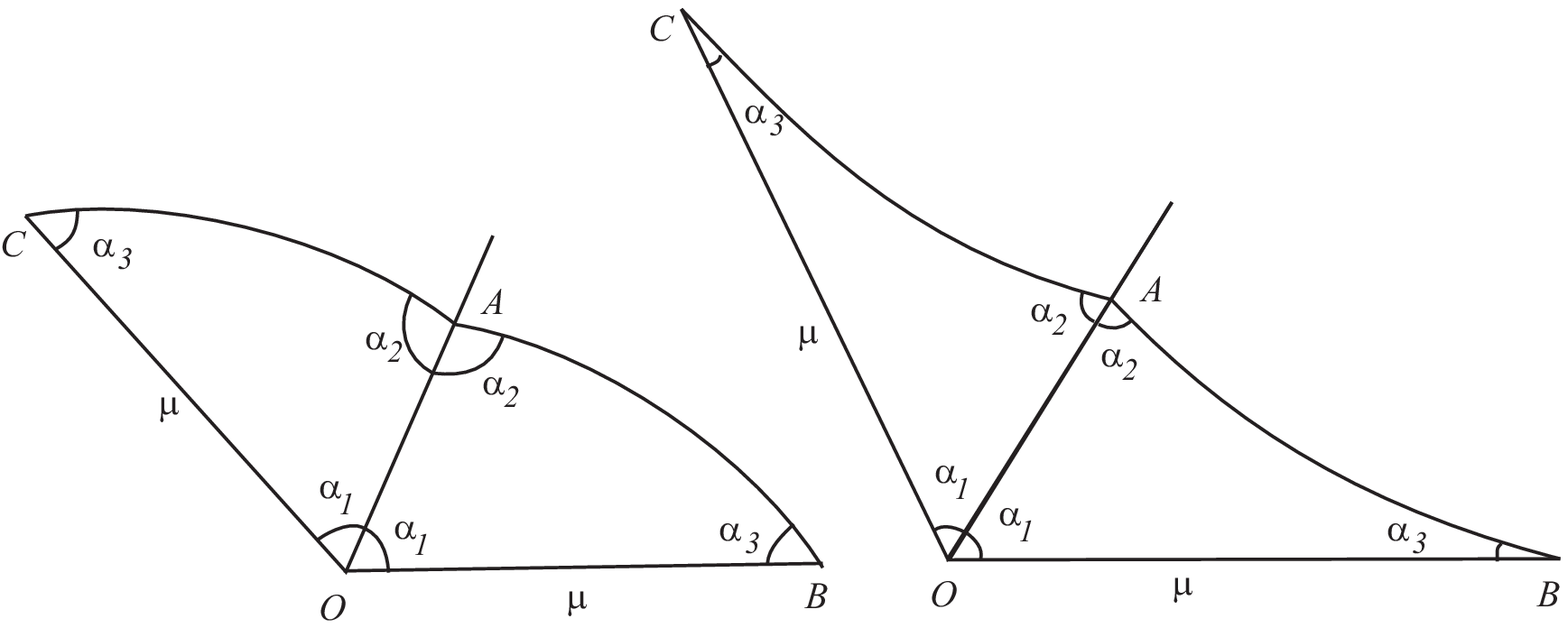,height=5cm}
\caption{}\label{fpqadoblealfas}\end{center}
\end{figure}

\begin{theorem}\label{tbase3}
There exists a  continuous family,  depending on three parameters, of geometric cone manifold structures,
\[
(O,0\, |\, \pi /\alpha_1,\pi /\alpha_2,\pi /\alpha_3 ),
 \]in the 2-sphere $S^{2}$ with three conic points with angles $(2\alpha_1,2\alpha_2, 2\alpha_3 )$ such that the geometry is
 \begin{itemize}
   \item hyperbolic for  $0\leq \alpha_1+\alpha_2+\alpha_3 <\pi$,
   \item Euclidean for  $\alpha_1+\alpha_2+\alpha_3 =\pi $, $\alpha_i>0$,
   \item spherical for
   \begin{itemize}
     \item the interior of the tetrahedon
   \[
   t_s=\left((\pi,0,0),(0,\pi,0),(0,0,\pi),(\pi ,\pi ,\pi )\right).
   \]
     \item the edges of $t_s$
     \begin{eqnarray*}
     l_1 &=& \left( (\pi,0,0),(\pi ,\pi ,\pi )\right) \\
     l_2 &=& \left( (0,\pi,0),(\pi ,\pi ,\pi )\right) \\
     l_3 &=& \left( (0,0,\pi),(\pi ,\pi ,\pi )\right)
     \end{eqnarray*}
   The faces of this tetrahedron define the limits of sphericity.
   \end{itemize}
 \end{itemize}
\end{theorem}
\begin{proof}
The cone manifold $(O,0\, |\, \pi /\alpha_1,\pi /\alpha_2,\pi /\alpha_3 )$ is obtained from the union of the two left (resp. right) triangles  in Figure \ref{fpqadoblealfas} by the isometric identifications $\overline{OB}\equiv \overline{OC}$ and $\overline{AB}\equiv \overline{AC}$, in the spherical (resp. Hyperbolic) case.

For points in the interior of the faces  $f_1=\left((0,\pi,0),(0,0,\pi),(\pi ,\pi ,\pi )\right)$, \newline $f_2=\left((\pi,0,0),(0,0,\pi),(\pi ,\pi ,\pi )\right)$ and $f_3=\left((\pi,0,0),(0,\pi,0),(\pi ,\pi ,\pi )\right)$  the geometry is Euclidean but $T_{(\alpha_1,\alpha_2,\alpha_3)}$ is no longer a triangle, therefore the geometric cone manifold structures do not exist.  Figure \ref{fregionesalfa}.
\end{proof}

 We are also interested in the particular cases with only two or one variable angles.

 \begin{figure}[h]
\begin{center}\epsfig{file=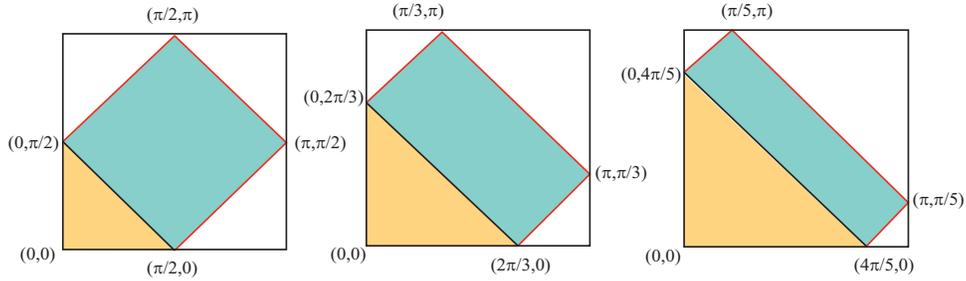,height=3.7cm}
\caption{Cases $a_1=2,3,5$}\label{dosalfas}\end{center}
\end{figure}

\begin{theorem}\label{tbase2}
For each natural number $a_1$, $1<a_1$, there exists a bi-parametric continuous family of geometric cone manifold structures,
\[(O,0\, |\, a_1,\pi /\alpha_2,\pi /\alpha_3 ),
 \]
 in the 2-sphere $S^{2}$, with three conic points with angles $(\frac{2\pi }{a_1},2\alpha_2, 2\alpha_3 )$ such that the geometry is
 \begin{itemize}
               \item hyperbolic for  $0\leq \alpha_2+\alpha_3 <\frac{a_1-1}{a_1}\pi$,
               \item Euclidean for  $\alpha_2+\alpha_3 =\frac{a_1-1}{a_1}\pi$,  $\alpha_i>0$,
               \item spherical for
               \begin{itemize}
               \item points in the interior of the rectangle, where
               \[ \left\{ \begin{array}{l}
\frac{a_1-1}{a_1}\pi < \alpha_2+\alpha_3 < \frac{a_1+1}{a_1}\pi, \, \text{and}  \\
\frac{1-a_1}{a_1}\pi < \alpha_2-\alpha_3 < \frac{a_1-1}{a_1}\pi.
\end{array}\right.  \]
             The edges of this rectangle
define the limits of sphericity for the 2-conemanifold $(O,0\, |\, a_1,\pi /\alpha_2,\pi /\alpha_3 )$.
\item The vertices $(\pi ,\pi /a_1)$ and $(\pi /a_1,\pi )$ of this rectangle  define the spherical orbifold $(O,0\, |\, a_1,a_1 )$ with two conic points.
    \end{itemize}
\end{itemize}
The case $a_1=1$  gives the continuous family of spherical cone manifold structures,
\[
(O,0\, |\, \pi /\alpha,\pi /\alpha ),
\] in the 2-sphere $S^{2}$, with two conic points with equal angle $2\alpha$, $0<\alpha \leq \pi$.
\end{theorem}
\begin{proof}
The geometry on a triangle $T_{\pi /a_1,\alpha_2,\alpha_3)}$ with a fixed angle $\pi /a_1$ and two variable angles
 $(\alpha_2,\alpha_3)$ corresponds to the points in the intersection of the cube in $\mathbb{R}^3$ with the plane $\alpha_1=\pi /a_1$. Figure \ref{dosalfas}.

The  spherical cone manifold structure, $(O,0\, |\, \pi /\alpha,\pi /\alpha )$, is defined by points in the edges with vertices $\left( (\pi,0,0),(\pi ,\pi ,\pi )\right)$, $\left( (0,\pi,0),(\pi ,\pi ,\pi )\right)$ and $\left( (0,0,\pi),(\pi ,\pi ,\pi )\right)$ in the faces of the cube.
\end{proof}

The case with two fixed angles $\pi /a_1,\,\pi /a_2$ and one variable angle
 $\alpha_3$ is the following.

\begin{theorem}\label{tbase1}
For each pair of natural numbers $(a_1,a_2)$, $1<a_1\leq a_2$ there exists a continuous family of geometric cone manifold structures, \[
(O,0\, |\, a_1,a_2,\pi /\alpha_3 ),
\]
in the 2-sphere $S^{2}$, with three conic points with angles $(\frac{2\pi }{a_1},\frac{2\pi }{a_2}, 2\alpha_3 )$ such that the geometry is hyperbolic for  $0<\alpha_3 <\alpha_{L}$, Euclidean for  $\alpha_{L}$ and spherical for $\alpha_{L}< \alpha_3 <\alpha_{U}$, where
\[2\,\alpha _{L}= 2\,\frac{a_2a_1-a_2-a_1}{a_2a_1}\pi \quad \quad 2\,\alpha_{U}=2\,\frac{a_2a_1-a_2+a_1}{a_2a_1}\pi
\]
are, respectively, the lower and the upper limits of sphericity for the 2-conemanifold $(O,0\, |\, a_1,a_2,\pi /\alpha_3 )$. The amplitude of this spherical interval, $[2\,\alpha_{L} ,2\,\alpha_{U}]$, is inversely proportional to $a_2$: $2\,\alpha_{U} -2\,\alpha_{L} =\frac{4\pi}{a_2}$.

The case $a_1=1$ gives the  spherical orbifold structures, $(O,0\, |\, 2\pi /a_2,2\pi /a_2)$, in the 2-sphere $S^{2}$, with two conic points with equal angle $2\pi /a_2$.
\end{theorem}
\begin{proof}
For each pair of natural numbers $(a_1,a_2)$, $1<a_1\leq a_2$ the region of sphericity is an interval. Since the Euclidean case comes from
\[
\frac{\pi }{a_1}+\frac{\pi }{a_2}+ \alpha_3 =\pi \quad \Leftrightarrow \quad \alpha_3 =(1-\frac{1 }{a_1}-\frac{1 }{a_2})\,\pi =\frac{a_2a_1-a_2-a_1}{a_2a_1}\, \pi \geq 0
\]
then $2\frac{a_2a_1-a_2-a_1}{a_2a_1}$ is the limit of hyperbolicity and also the lower limit of sphericity. The obtain the upper limit of sphericity observe that it is limited by the inequalities given by $p_3,\,p_5,\,p_7$:

\begin{equation}\label{limitesuno}
\begin{array}{ll}
\frac{\pi}{a_1}- \frac{\pi}{a_2}+\alpha_3<\pi &\quad \Leftrightarrow \quad \alpha_3 < \frac{a_1a_2-a_2+a_1}{a_1a_2}\,\pi \\
-\frac{\pi}{a_1}+ \frac{\pi}{a_2}+\alpha_3<\pi &\quad \Leftrightarrow \quad \alpha_3 <  \frac{a_1a_2+a_2-a_1}{a_1a_2}\,\pi \\
\frac{\pi}{a_1}+ \frac{\pi}{a_2}-\alpha_3<\pi &\quad \Leftrightarrow \quad \alpha_3 >  \frac{-a_1a_2+a_2+a_1}{a_1a_2}\, \pi \leq 0
\end{array}
\end{equation}
The third inequality in (\ref{limitesuno}) gives no restriction because $1<a_1\leq a_2$.
Since $a_1\leq a_2$, the first two inequalities in (\ref{limitesuno}) reduce to
\[
\alpha_3 <  \frac{a_1a_2-a_2+a_1}{a_1a_2}\, \pi.
\]
Hence $2\frac{a_1a_2-a_2+a_1}{a_1a_2}\, \pi$ is the  upper limit of sphericity.

\end{proof}

\section{The geometric structure on $M\in \mathcal{F}$ }\label{s3}

We consider $\mathbb{R}^3$ as the space of coordinates $(\beta_1,\beta_2,\beta_3)$, and we obtain the regions for the possible geometries in  the conemanifold structure in $M$ with cone angles  $\beta_i$ along the $(a_i,b_i)$-fibres, $i=1,2,3$. If the $(a_i,b_i)$-fibre does not belong to the singular set, the angle $\beta_i$ is $2\pi$.

In the formulation of the next theorems we write in brackets the case $e(M)= 0$.

\begin{figure}[h]
\begin{center}\epsfig{file=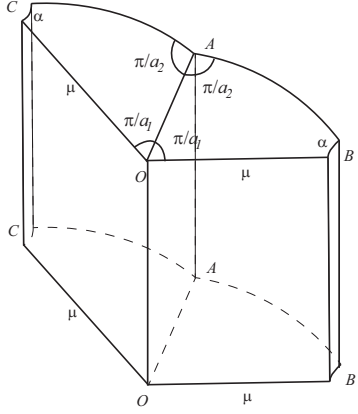,height=5.5cm}
\caption{The prism $D$}\label{fprisma}\end{center}
\end{figure}
\begin{figure}[h]
\begin{center}\epsfig{file=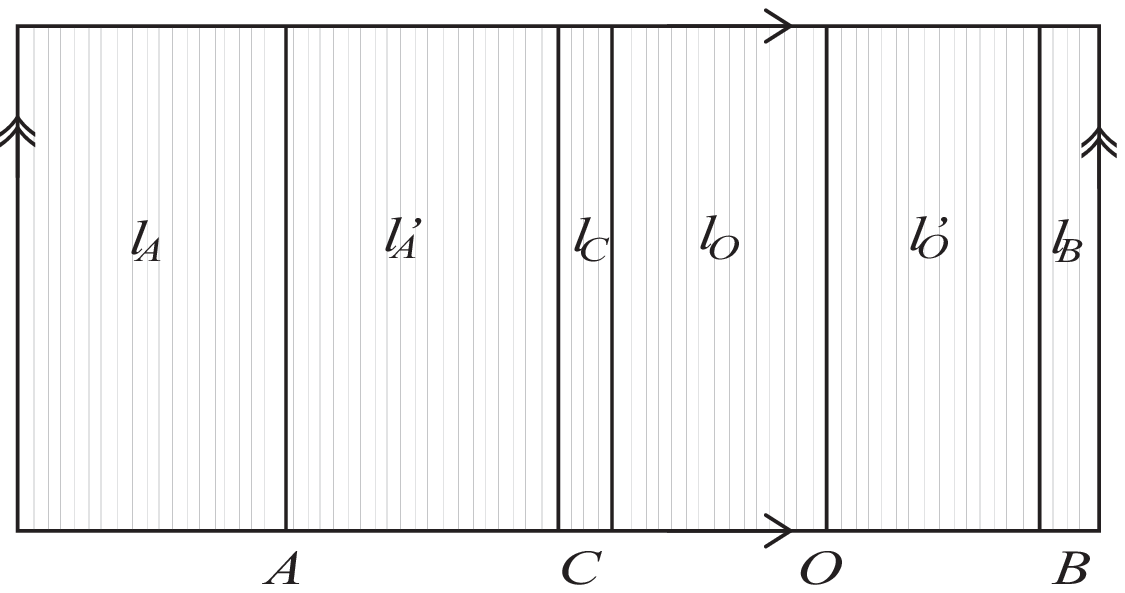,height=5cm}
\caption{The boundary of $\widehat{D}$}\label{ftoro}\end{center}
\end{figure}

\begin{theorem}\label{t3fibras}
Let $M$ be a Seifert manifold
\begin{eqnarray*}
M&=&\left\langle O,o,0\, |\, b;(a_{1},b_{1}),(a_{2},b_{2}),(a_{3},b_{3})\right\rangle = \\
&=& \left\langle O,o,0\, |\, 0;(a_{1},b_{1}),(a_{2},b_{2}),(a_{3},b_{3}+ba_3)\right\rangle
\end{eqnarray*}
such that either $(a_{i},b_{i})$ are coprime integers, with $0<b_{i}<a_{i}>1$, or $(a_{i},b_{i})=(1,0)$. If the Euler number $e(M)\neq 0$ (resp. $e(M)=0$), then $M$
has geometric conemanifold structures $\left( M,(\beta_1,\beta_2,\beta_3) \right)$ with the singularity $L$, if any, is along the $(a_i,b_i)$-fibres.  The geometry can be $\widetilde{SL(2,\mathbb{R}})$, Nil or spherical (resp. $H^2\times \mathbb{R}$, Euclidean or $S^2\times \mathbb{R}$).

The geometry is $\widetilde{SL(2,\mathbb{R}})$ (resp. $H^2\times \mathbb{R}$) for
\[
0\leq \beta_1a_2a_3+\beta_2a_1a_3+\beta_3a_1a_2<2a_1a_2a_3 \pi
\]

The geometry is Nil (resp. Euclidean) for
\[
\beta_1a_2a_3+\beta_2a_1a_3+\beta_3a_1a_2=2a_1a_2a_3 \pi ,\quad \beta_i>0
\]

The geometry is spherical (resp. $S^2\times \mathbb{R}$) for points in the interior of the tetrahedron $t_s$ in $\mathbb{R}^3$
\[
t_s= \left( (2a_1\pi ,0,0),(0,2a_2\pi ,0),(0,0,2a_3\pi ),(2a_1\pi ,2a_2\pi ,2a_3\pi )\right)
 \]
 and for points in the edges
 \[
 \begin{array}{l}
 l_1 = \left( (2a_1\pi ,0,0),(2a_1\pi ,2a_2\pi ,2a_3\pi )\right)\\
 l_2 = \left( (0,2a_2\pi ,0),(2a_1\pi ,2a_2\pi ,2a_3\pi )\right)\\
 l_3 = \left( (0,0,2a_3\pi ),(2a_1\pi ,2a_2\pi ,2a_3\pi )\right).
 \end{array}
 \]
\end{theorem}

\begin{proof}
Suppose $e(M)\neq 0$. By a similar construction as the one made for the trefoil knot in \cite{LM2015}, one can define geometric cone manifold structures in the Seifert manifold $M$,
such that the 2-conemanifold structure in the orbit space is
\[
(O,0\, |\, \pi /\alpha_1,\pi /\alpha_2,\pi /\alpha_3 )
 \]
 and the singular set $L$ in $M$ is the set of exceptional fibres, where the angle  around the $(a_i,b_i)$-fibre is
\begin{equation}\label{ebetaalfa}
\beta_i = a_i 2\alpha.
\end{equation}

Let $\mathbf{B}$ be the union of the two  triangles depicted in Figure \ref{fpqadoblealfas} minus a neighborhood of the vertices $C$ and $B$ with angle $\alpha$, provided with the suitable geometry (hyperbolic, Euclidean or spherical). Consider a fibred prism $D$ with orbit space $\mathbf{B}$ in the corresponding geometry ($\widetilde{SL(2,\mathbb{R})}$, Nil or $S^3$ geometry), Figure \ref{fprisma}. Identify botton and top by the isometric translation along the fibres by $2\pi$. The boundary of the resulting solid torus $\widehat{D}$ is divided into six fibred annuli, Figure \ref{ftoro}. Two of them, say $l_A$ and $l'_A$, intersect  each other along the fibre over the vertex $A$: $A\times S^{1}$. The fibres $l_O$ and $l'_O$,  intersect each other along the fibre over the vertex $O$: $O\times S^{1}$. The remaining two, say $l_C$ and $l_B$, are the product of the boundary of the neighborhood of vertices $C$ and $B$ cross $S^1$. Identify $l_O$ and $l'_O$ by an isometry, $g(a_1,b_1)$, which projects into the orbit space upon a rotation of angle $2\alpha_i$ around de point $O$, leaving invariant the fibre $O\times S^{1}$ in such a way that the fibration induced on the quotient has the invariant fibre as  an exceptional fibre of type $(a_1,b_1)$. The result is a solid torus with a geometric Seifert structure with an exceptional fibre of type $(a_1,b_1)$.  Next, identify analogously $l_A$ and $l'_A$ by an isometry, $g(a_2,b_2)$, in such a way that the fibration induced on the quotient has the invariant fibre $A\times S^{1}$ as  an exceptional fibre if type $(a_2,b_2)$. We have obtained a geometric manifold $N$, with boundary  the fibred torus $l_C \cup l_B$. Attach a solid fibred torus $T_{(a_{3},b_{3}+ba_3)}$, whose core is a exceptional fibre of type $(a_{3},b_{3}+ba_3)$. The meridian of $T_{(a_{3},b_{3}+ba_3)}$ is glued to a curve made up of  two ``horizontal'' segments. Extend the geometry in $N$ to the attached $T_{(a_{3},b_{3}+ba_3)}$. The result is our Seifert manifold $M$ with geometry. The exceptional fibre $(a_{3},b_{3}+ba_3)$ projects into the orbit space to the point $C\equiv B$ with an angle $2\alpha_3$. Then the angle around the $(a_i,b_i)$-fibre  is $\beta_i =2\alpha_i a_i$, $i=1,2,3$.

If $e(M)=0$, consider the product $S^2\times S^1$. Endow the factor $S^2$ with the geometric conemanifold structure $(O,0\, |\, \pi /\alpha_1, \pi /\alpha_2,\pi /\alpha_3)$, given in Theorem \ref{tbase3}. Let $A_i$ be the cone point with angle $\alpha_i$.
Next, make surgery along each fibre $A_i\times S^1$ such that it becomes a exceptional fibre of type $(a_i,b_i)$ for $i=1,2$ and $(a_{3},b_{3}+ba_3)$ for $i=3$.
The resulting manifold is our Seifert manifold $M$ with the corresponding product geometry.

Therefore, in both cases, the issue is a geometric Seifert conemanifold with singularity, if any, along the $(a_i,b_i)$-fibres where the angle is
\[
\beta_i=2a_i\alpha_i .
\]
This relation between $\alpha_i$ and $\beta_i$ and Theorem \ref{tbase3} gives the values of the angles $\beta_i$ for every possible geometry.

Consider the 3-dimensional space with coordinates $(\beta_1,\beta_2,\beta_3)$. The points in this space defining the  values for $\beta_i$ thus obtained belong to the cube $[0, 2a_1\pi]\times  [0,2a_2\pi]\times [0,2a_3\pi]$. The regions for the different geometries are obtained by appropriated scaling of the regions in Figure \ref{fregionesalfa}.
\end{proof}

\begin{theorem}\label{t2fibras}
Let $M$ be a Seifert manifold
\begin{eqnarray*}
M&=&\left\langle O,o,0\, |\, b;(a_{1},b_{1}),(a_{2},b_{2}),(a_{3},b_{3})\right\rangle = \\
&=& \left\langle O,o,0\, |\, 0;(a_{1},b_{1}),(a_{2},b_{2}),(a_{3},b_{3}+ba_3)\right\rangle
\end{eqnarray*}
such that either $(a_{i},b_{i})$ are coprime integers, with $0<b_{i}<a_{i}>1$, or $(a_{i},b_{i})=(1,0)$. If the Euler number $e(M)\neq 0$ (resp. $e(M)=0$), then $M$
has geometric conemanifold structures $\left( M,(\beta_2,\beta_3) \right)$ with the singularity $L$, if any, is along the $(a_i,b_i)$-fibres, $i=2,3$.  The geometry can be $\widetilde{SL(2,\mathbb{R}})$, Nil or spherical (resp. $H^2\times \mathbb{R}$, Euclidean or $S^2\times \mathbb{R}$).

The geometry is $\widetilde{SL(2,\mathbb{R}})$ (resp. $H^2\times \mathbb{R}$) for
\[
0\leq \beta_2a_1a_3+\beta_3a_1a_2<2(a_1-1)a_2a_3\pi
\]

The geometry is Nil (resp. Euclidean) for
\[
\beta_2a_1a_3+\beta_3a_1a_2=2(a_1-1)a_2a_3 \pi>0
\]

The geometry is spherical (resp. $S^2\times \mathbb{R}$) for points in the interior of the rectangle in $\mathbb{R}^2$ with coordinates $(\beta_2,\beta_3)$, where
\[
\left\{\begin{array}{c}
2(a_1-1)a_2a_3 \pi <\beta_2a_1a_3+\beta_3a_1a_2<2(a_1+1)a_2a_3 \pi \, \text{and} \\
2(1-a_1)a_2a_3 \pi <\beta_2a_1a_3-\beta_3a_1a_2<2(a_1-1)a_2a_3 \pi
\end{array}\right. ,
\]

and also for $(\beta_2 , \beta_3 )=( 2\pi ,\frac{2\pi a_3}{a_1})$ or $(\beta_2 , \beta_3 )=( \frac{2\pi a_2}{a_1} ,2\pi )$.

The geometry is also spherical ($S^2\times \mathbb{R}$) for the three families of conemanifold structures in $M$ with cone angles $(2\alpha a_1,2\alpha a_2,2\pi a_3)$,  $(2\alpha a_1,2\pi a_2,2\alpha a_3)$ and $(2\pi a_1,2\alpha a_2,2\alpha a_3)$, $0<\alpha \leq \pi$, along the fibres $(a_1,a_2,a_3)$.
\end{theorem}

\begin{proof}
The proof is analogous to the proof of the above Theorem \ref{t3fibras} but endowing the factor $S^2$ with the geometric conemanifold structure $(O,0\, |\, a_1, \pi /\alpha_2,\pi /\alpha_3)$ given in Theorem \ref{tbase2}.
\end{proof}

\begin{theorem}\label{t1fibra}
Let $M$ be a Seifert manifold
\begin{eqnarray*}
M&=&\left\langle O,o,0\, |\, b;(a_{1},b_{1}),(a_{2},b_{2}),(a_{3},b_{3})\right\rangle = \\
&=& \left\langle O,o,0\, |\, 0;(a_{1},b_{1}),(a_{2},b_{2}),(a_{3},b_{3}+ba_3)\right\rangle
\end{eqnarray*}
such that either $(a_{i},b_{i})$ are coprime integers, with $0<b_{i}<a_{i}>1$, or $(a_{i},b_{i})=(1,0)$. If the Euler number $e(M)\neq 0$ (resp. $e(M)=0$), then $M$
has geometric conemanifold structures $\left( M,(\beta_3) \right)$ with the singularity $L$, if any, is the $(a_3,b_3)$-fibre.  The geometry can be $\widetilde{SL(2,\mathbb{R}})$, Nil or spherical (resp. $H^2\times \mathbb{R}$, Euclidean or $S^2\times \mathbb{R}$).

The geometry is $\widetilde{SL(2,\mathbb{R}})$ (resp. $H^2\times \mathbb{R}$) for
\[
0\leq \beta_3< 2\frac{a_3(a_2a_1-a_2-a_1)}{a_2a_1}\pi=\beta_L
\]

The geometry is Nil (resp. Euclidean) for
\[
\beta_3=\beta_L=2\frac{a_3(a_2a_1-a_2-a_1)}{a_2a_1}\pi
\]

The geometry is spherical (resp. $S^2\times \mathbb{R}$) for
\[
\beta_L=2\frac{a_3(a_2a_1-a_2-a_1)}{a_2a_1}\pi <\beta_3 <2\frac{a_3(a_2a_1-a_2+a_1)}{a_2a_1}\pi =\beta_U.
\]
The values $\beta_L$ and $\beta_U$ are the lower and upper limits of sphericity.
\end{theorem}

\begin{proof}
The proof is analogous to the proof of the above Theorem \ref{t3fibras} but endowing the base $S^2$ with the geometric conemanifold structure $(O,0\, |\, a_1, a_2,\pi /\alpha_3)$ given in Theorem \ref{tbase1}.
\end{proof}

\begin{corollary}
The cone angle around the singular set $L$ in the Nil cone manifold structure in the Seifert manifold
$(M,\beta_3)$, is always a rational multiple of $\pi$.
\end{corollary}

\begin{proof}
The value of the cone angle is $\beta _{L}= 2\frac{a_3(a_2a_1-a_2-a_1)}{a_2a_1}\pi $
\end{proof}

\begin{corollary}
The quotient $\beta_{U}/\beta_{L}$ between the upper and lower sphericity limits for the conemanifold structures $(M,\beta_3)$ in the manifold $M$ does not depend on the third fibre $a_{3}$.
\end{corollary}
\begin{proof}
It follows from Theorem \ref{t1fibra} that
\[
\frac{\beta_{U}}{\beta_{L}}= \frac{a_2a_1-a_2+a_1}{a_2a_1-a_2-a_1}
\]
\end{proof}
Observe that the relative position of $2\pi$ respect to $\beta_{L}$ and $\beta_{U}$ define the geometry of the Seifert manifold $M$, where the fibres are geodesics:

$\widetilde{SL(2,\mathbb{R})}$ (resp. $H^2\times \mathbb{R}$) geometry $\Leftrightarrow$ $2\pi < \beta_L$

Nil (resp. Euclidean) geometry $\Leftrightarrow$ $\beta_L=2\pi$,

Spherical (resp. $S^2\times \mathbb{R}$) geometry $\Leftrightarrow$ $\beta_L<2\pi < \beta_U$.

\section{$\mathcal{F}_1$: Manifolds with finite fundamental group}\label{s4}

The Seifert manifolds with finite fundamental group and orbit space $S^{2}$ have at most three exceptional fibres and Euler number $e(M)\neq 0$. All of them are spherical manifolds, therefore for all the conemanifold structures with singular set only one exceptional or general fibre, $2\pi$ lies between the limits of sphericity:  $\beta_{L} <2\pi <\beta_{U}$.
Manifolds with less or equal than two exceptional fibres are lens spaces. If the manifold has three exceptional fibres, they are of multiplicity $(2,2,n),\, (2,3,3),\, (2,3,4),\, (2,3,5)$. Because it is an interesting particular case,  we first study  the case of the Hopf fibration, although it could be included in the general study for Seifert fibrations in lens spaces.
\subsection{$S^3$ with no exceptional fibres: the Hopf fibration}\hspace*{\fill} \\

The Hopf fibration on $S^3$ is the Seifert manifold
\[
M=(O,o,0\, |\, -1;)=\left\langle O,o,0\, |\, -1;(1,0),(1,0),(1,0) \right\rangle
\]
This Seifert structure in $S^3$ has  conemanifold structures $(S^3,(\beta_1,\beta_2,\beta_3))$ with singularity three fibres of the Hopf fibration with angle $\beta_i$ on the fibre $a_i$, whose geometry depends on the region containing the point
\[
(\alpha_1,\alpha_2,\alpha_3)=\left( \frac{\beta_1}{2}, \frac{\beta_2}{2},\frac{\beta_3}{2}\right)
\]
 in Figure \ref{fregionesalfa}, according to Theorem \ref{t3fibras}. The three geometries (Spherical, Nil and $\widetilde{SL(2,\mathbb{R}})$) are possible.

The possibilities with only two fibres as singular set, the Hopf link, correspond to points in the faces of the cube in Figure \ref{fregionesalfa}
\begin{eqnarray*}
  \Pi_1:  \alpha_1 &=& \pi ,   \\
  \Pi_2:  \alpha_2 &=& \pi \,\, \text{and}  \\
  \Pi_3:  \alpha_3 &=& \pi .
\end{eqnarray*}

\begin{remark}
There are  conemanifold spherical structures $(S^3,(\beta,\beta))$, with singular set the Hopf link, where  $0<\beta \leq 2\pi$. The fibres of the Hopf fibration are geodesics of the geometric structure.
\end{remark}

\begin{remark}
The Seifert manifold $S^3=(O,o,0\, |\, -1;)$ admits no geometric Seifert conemanifold structure with singular set one fibre of this (Hopf) fibration.
\end{remark}

\subsection{Seifert manifold structures in lens spaces}\hspace*{\fill} \\

The Seifert manifold
\[
M=(O,o,0\, |\, b;(a_{1},b_{1}),(a_{2},b_{2}))=\left\langle O,o,0\, |\, b;(a_{1},b_{1}),(a_{2},b_{2}),(a_3,b_3) \right\rangle
\]
where $(a_3,b_3)=(1,0)$ and with Euler number number
\[
e=- \frac{ba_1a_2+a_1b_2+a_2b_1}{a_1a_2} \neq 0\quad \Rightarrow \quad m:=ba_1a_2+a_1b_2+a_2b_1\neq 0,
\]
is the lens space $L(m,ra_2+sb_2)$ where $-ra_1+s(ba_1+b_1)=1$ (\cite{O1972}. The orientation in \cite{O1972} is the opposite to the standard one used currently.).

We consider the following three cases
\begin{enumerate}
  \item Two exceptional fibres. Here $(a_{i},b_{i})$ are coprime integers, $0<b_{i}<a_{i}$, $i=1,2$, $1<a_1\leq a_2$, $(a_3,b_3)=(1,0)$;
  \item One exceptional fibre. Here $(a_{1},b_{1})$ are coprime integers, $0<b_{1}<a_{1}$, \, $(a_{2},b_{2})=(a_3,b_3)=(1,0)$;
  \item No exceptional fibres, $(a_{i},b_{i})=(1,0)$,  $i=1,2,3$.
\end{enumerate}

According to Theorem \ref{t3fibras}, the geometry of the  conemanifold structure $(M,(\beta_1,\beta_2,\beta_3))$, where $\beta_i$ is the angle  on the $(a_i,b_i)$-fibre,  depends on the region in Figure \ref{fregionesalfa} to which the point
\[
(\alpha_1,\alpha_2,\alpha_3)=\left( \frac{\beta_1}{2a_1}, \frac{\beta_2}{2a_2},\frac{\beta_3}{2a_3}\right)
\]
belongs.  The three geometries (Spherical, Nil and $\widetilde{SL(2,\mathbb{R}})$) are possible.

According to Theorem \ref{t2fibras}, the geometry when the singular set is the union of  two  $(a_i,b_i)$-fibres is given by a point $(\alpha_1,\alpha_2,\alpha_3)$ belonging to the intersection of the cube in Figure \ref{fregionesalfa} with one of the following planes
\begin{eqnarray*}
  \Pi_1:  \alpha_1 &=& \pi /a_1 , \\
  \Pi_2:  \alpha_2 &=& \pi /a_2  \\
  \Pi_3:  \alpha_3 &=& \pi .
\end{eqnarray*}
In case (1) (two exceptional fibres), the intersection of the cube with $\Pi_1$ or $\Pi_2$ yields  regions, with shape as in Figure \ref{dosalfas}, where the three geometries (Spherical, Nil and $\widetilde{SL(2,\mathbb{R}})$) are possible. The singularity of the Seifert conemanifold structure is the union of  an exceptional fibre and a general fibre.  The plane $\Pi_3$ intersects the cube in one of its faces. Therefore the singular set of the associated conemanifold structures is the union of the two exceptional fibres, in such a way that
the continuous family only depends on one parameter, say $\alpha$. The elements of this continuous family  are the spherical conemanifold structures  $(M,(\beta_1,\beta_2))$, where $\beta_i=2\alpha a_i$, $0<\alpha \leq \pi$, (Theorem \ref{t2fibras}).

In case (2) (one exceptional fibre $(a_1, b_1)$ ($a_2=a_3=1$)), the intersection of the cube with $\Pi_1$ yields  regions as in Figure \ref{dosalfas}, where the three geometries are possible. The intersection of the cube with $\Pi_2$ or $\Pi_3$ are faces of the cube. Then, the singular set of the associated conemanifolds structures is the union of  the  exceptional fibre and a general fibre. As before,
this continuous family only depends on one parameter, say $\alpha$. The elements of this continuous family  are the spherical conemanifold structures $(M,(2\alpha a_1,2\alpha))$,  $0<\alpha \leq \pi$.

In case (3) (no exceptional fibres), ($a_1=a_2=a_3=1$), the intersection of the cube with $\Pi_1$, $\Pi_2$ or $\Pi_3$, are faces of the cube. There exist three analogous continuous families depending on one parameter, say $\alpha$, of conemanifolds structures whose singular set is the union  two general fibres. The elements of this continuous family  are the spherical conemanifold structures  $(M,(2\alpha,2\alpha))$, $0<\alpha \leq \pi$.

\begin{remark}
In particular, when $M$ is $S^3$ with a Seifert structure with two exceptional fibres $(a_{1},b_{1})$ and $(a_{2},b_{2})$, we obtain a continuous family of spherical conemanifold structures $(S^3,(\beta_1,\beta_2))$, where the singular set is the Hopf link and $\beta_i=2\alpha a_i$, $0<\alpha \leq \pi$. The fibres, toroidal knots, are geodesics of the geometry.
\end{remark}

The geometry when the singular set is only one $(a_i,b_i)$-fibre is given by a point $(\alpha_1,\alpha_2,\alpha_3)$ belonging to the intersection of the cube in Figure \ref{fregionesalfa} with one of the following lines
 \[
 \begin{array}{ll}
  \Gamma_1:  \alpha_2 = \pi /a_2, & \quad \alpha_3 = \pi \\
  \Gamma_2:  \alpha_1 = \pi /a_1, & \quad\alpha_3 = \pi  \\
  \Gamma_3:  \alpha_1 = \pi /a_1, & \quad\alpha_2 = \pi /a_2.
\end{array}
 \]
 If the line is contained in the interior of the cube, a continuous family of conemanifold structures exists. This is the case of $\Gamma_3$, when $a_1\neq 1\neq a_2$. If the line is in the interior of a face of the cube, only an orbifold structure exists. For instance, $\Gamma_1$ when $a_2\neq 1$, yields the orbifold structure $(M,a_2)$ whose singular set is the $(a_1,b_1)$-fibre. If the line is an edge of the cube, there is no conemanifold structure. In this case there is only the spherical manifold structure $M$.

 This implies that in a Seifert manifold with two exceptional fibres, geometric conemanifold structures  whose singular set is composed of just one fibre, are possible if and only if the singular set is a general fibre.

Let  $\left(L(m,ra_2+sb_2),\beta \right)$ be a conemanifold structure  whose singularity is a general fibre of the Seifert structure $(O,o,0\, |\, b;(a_{1},a_{1}),(a_{2},b_{2}))$. By Theorem \ref{t1fibra},
its limits of sphericity are the angles
\begin{equation}
   \beta _{L}= 2\frac{a_2a_1-a_2-a_1}{a_2a_1}\pi , \quad  \beta _{U}= 2\frac{a_2a_1-a_2+a_1}{a_2a_1}\pi .
\end{equation}
The geometry is spherical for
\begin{equation}
\beta_{L} <\beta <\beta_{U};
\end{equation}
Nil geometry for $\beta =\beta _{L} $; and $\widetilde{SL(2,\mathbb{R}})$ geometry for
\begin{equation}
0\leq \beta <\beta _{L} .
\end{equation}

\subsection{$\{a_1,a_2,a_3\}=\{2,2,n\}$: prism manifold}\hspace*{\fill} \\

The manifold
\[
M=(O,o,0\, |\, b;(2,1),(2,1),(n,b_{3}))=\left\langle O,o,0\, |\, -1;(2,1),(2,1),(n,(b+1)n+b_3)\right\rangle
\]
where $(n,b_{3})$ are coprime integers, $0<b_{3}<n$, $n\geq 2$ and
\[
e= -\frac{bn+n+b_3}{n}  \neq 0 \quad \Rightarrow \quad m:=(b+1)n+b_3\neq 0
\]
is called the \emph{prism manifold}  $P_{n,m}$. It is the manifold obtained from a prism with base a regular polygon of $4m$ edges by identifying the botton with the top of the prism by a $(\pi/2m)$ right rotation, and each lateral face of the prism is identify with the opposite face by a $(\pi /2)$ right rotation.

The manifold $P_{n,1}=(O,o,0\, |\, -1;(2,1),(2,1),(n,1))$ is the manifold of spherical tessellations $M(Sn22)$ (\cite{M1987}).
The simplest example is the \emph{quaternionic manifold}
\[
P_{(2,1)}=(O,o,0\, |\, -1;(2,1),(2,1),(2,1))
 \]
 whose fundamental group is the quaternion group of 8 elements. The manifold is the quotient of $S^{3}$ by the action of this group and is has a cube as a fundamental domain. Each face of the cube is identify with the opposite by a $(\pi /2)$ right rotation.

This prism manifold has  conemanifold structures $(M,(\beta_1,\beta_2,\beta_3))$ with angle $\beta_i$ on the $(a_i,b_i)$-fibre. According to Theorem \ref{t3fibras}, the  geometry of $(M,(\beta_1,\beta_2,\beta_3))$ depends on the regions in Figure \ref{fregionesalfa} to which the point
\[
(\alpha_1,\alpha_2,\alpha_3)=\left( \frac{\beta_1}{4}, \frac{\beta_2}{4},\frac{\beta_3}{2n}\right)
\]
belongs. The three geometries (Spherical, Nil and $\widetilde{SL(2,\mathbb{R})}$) are possible.

The possibilities with only two fibres as singular set are obtained by the intersection of the cube in Figure \ref{fregionesalfa} with the planes
\begin{eqnarray*}
  \Pi_1:  \alpha_1 &=& \pi /4 , \\
  \Pi_2:  \alpha_2 &=& \pi /4  \,\, \text{and}\\
  \Pi_3:  \alpha_3 &=& \pi /2n
\end{eqnarray*}
 giving  regions as in Figure \ref{dosalfas}, where the three geometries (Spherical, Nil and $\widetilde{SL(2,\mathbb{R}})$) are possible.

The  conemanifold structure $\left(P_{(n,m)},\beta \right)$, with the fibre of order $n$  of the above Seifert structure as singular set has the following limits of sphericity
\begin{equation}
   \beta _{L}= 0 , \quad  \beta _{U}= 2n\pi .
\end{equation}
The geometry is spherical for
\begin{equation}
\beta_{L} <\beta <\beta_{U}= 2n\pi ;
\end{equation}
Nil geometry for $\beta =\beta _{L}=0 $; and there are no $\widetilde{SL(2,\mathbb{R})}$ geometric conemanifolds.

\begin{remark}
The exterior of the fibre of order $n$ in the prism manifold $P_{(n,m)}$ has a complete Nil geometry. There are no $\widetilde{SL(2,\mathbb{R})}$ geometric conemanifold structures on $P_{(n,m)}$ with the fibre of order $n$ as singular set.
\end{remark}

The prism manifold has also   conemanifold structures $\left(P_{(n,m)},\beta' \right)$, whose singularity is a $(2,1)$-fibre  of the above Seifert structure, with the following limits of sphericity
\begin{equation}
   \beta' _{L}= 2(n-2)\pi , \quad  \beta' _{U}= 2(n+2)\pi .
\end{equation}
The geometry is spherical for
\begin{equation}
\beta'_{L}= 2(n-2)\pi  <\beta' <\beta'_{U}= 2(n+2)\pi;
\end{equation}
Nil geometry for $\beta' =\beta _{L}=2(n-2)\pi  $; and $\widetilde{SL(2,\mathbb{R})})$ geometry for
\begin{equation}
0\leq \beta' <\beta' _{L}=2(n-2)\pi .
\end{equation}

In particular, if $n=2$, it follows that there are no $\widetilde{SL(2,\mathbb{R})}$ geometric conemanifold structures $\left(P_{(2,m)},\beta \right)$ with singular set an exceptional fibre.

\subsection{$\{a_1,a_2,a_3\}=\{2,3,3\}$}\hspace*{\fill} \\

The manifold
\[
M=(O,o,0\, |\, b;(2,1),(3,b_2),(3,b_{3}))=\left\langle O,o,0\, |\, -1;(2,1),(3,b_2),(3,(b+1)3+b_3)\right\rangle
\]
where $b_{2}$ and $b_{3}$ are 1 or 2,  and
\[
e= - \frac{6b+3+2(b_2+b_3)}{6}  \neq 0 \quad \Rightarrow \quad m:=6b+3+2(b_2+b_3)\neq 0.
\]
is denoted by $T(m)$. Observe that the odd integer $m$ determines $b$ and $b_2,\, b_3$ up to order.

The simplest case is $T(1)=(O,o,0\, |\, -1;(2,1),(3,1),(3,1))$ whose fundamental group is the binary tetrahedral group $T^*$ of order 24, and it is the manifold of spherical tessellations $M(S332)$ (\cite{M1987}).

The Seifert manifold $T(m)$ has geometric conemanifolds  structures whose singular set is composed by three or two exceptional fibres  according to Theorem \ref{t3fibras} and Theorem \ref{t2fibras} respectively. The three geometries (Spherical, Nil and $\widetilde{SL(2,\mathbb{R}})$) are possible.

The  conemanifold structures $\left( T(m),\beta \right) $, whose singularity is a $(3,1)$-fibre of the above Seifert structure, have the following limits of sphericity
\begin{equation}
   \beta _{L}= \pi , \quad  \beta _{U}= 5\pi .
\end{equation}
The geometry is spherical for
\begin{equation}
 \pi <\beta <  5\pi ;
\end{equation}
Nil geometry for $\beta =\pi $; and  $\widetilde{SL(2,\mathbb{R})}$ geometry for $\beta < \pi $. Therefore $\left( T(m),\pi \right) $ is a Nil orbifold whose singular set is a $(3,1)$-fibre with label 2.

The  conemanifold structures $\left(T(m),\beta \right)$, whose singularity is the $(2,1)$-fibre  of the above Seifert structure, have the following limits of sphericity
\begin{equation}
   \beta' _{L}= \frac{8}{3}\pi , \quad  \beta' _{U}= 4\pi .
\end{equation}
The geometry is spherical for
\begin{equation}
\frac{8}{3}\pi <\beta <  4\pi ;
\end{equation}
Nil geometry for $\beta =\frac{8}{3}\pi $; and  $\widetilde{SL(2,\mathbb{R})}$ geometry for $\beta < \frac{8}{3}\pi $.

\subsection{$\{a_1,a_2,a_3\}=\{2,3,4\}$}\hspace*{\fill} \\

The manifold
\[
M=(O,o,0\, |\, b;(2,1),(3,b_2),(4,b_{3}))=\left\langle O,o,0\, |\, -1;(2,1),(3,b_2),(4,(b+1)4+b_3)\right\rangle
\]
where $b_{2}$ is 1 or 2, and $b_{3}$ is 1, 2 or 3,  and
\[
e=- \frac{12b+6+4b_2+3b_3}{12}  \neq 0 \quad \Rightarrow \quad m:=12b+6+4b_2+3b_3\neq 0.
\]
is  denoted  by $O(m)$. Observe that the integer $m$, determines $b,\, b_2,\, b_3$.

The simplest case is $O(1)=(O,o,0\, |\, -1;(2,1),(3,1),(4,1))$ whose fundamental group is the binary octahedral group $O^*$ of order 48, and it is the manifold of spherical tessellations $M(S432)$ (\cite{M1987}).

The Seifert manifold $O(m)$ has geometric conemanifold  structures, whose singular set is composed by three or two exceptional fibres, according to Theorem \ref{t3fibras} and Theorem \ref{t2fibras}. The three geometries (Spherical, Nil and $\widetilde{SL(2,\mathbb{R}})$) are possible.

The  conemanifold structures $\left(O(m),\beta \right)$ whose singular set is  the $(4,b_3)$-fibre  of the above Seifert structure, have the following limits of sphericity
\begin{equation}
   \beta _{L}= \frac{4}{3}\pi , \quad  \beta _{U}= \frac{20}{3}\pi .
\end{equation}
The geometry is spherical for
\begin{equation}
\frac{4}{3} \pi <\beta < \frac{20}{3} \pi ;
\end{equation}
Nil geometry for $\beta =\frac{4}{3}\pi $; and  $\widetilde{SL(2,\mathbb{R}})$ geometry for $\beta < \frac{4}{3}\pi $.

The  conemanifold structures $\left(O(m),\beta \right)$, whose singular set is  the $(3,b_2)$-fibre of the above Seifert structure, have the following limits of sphericity
\begin{equation}
   \beta _{L}= \frac{3}{2}\pi , \quad  \beta _{U}= 3\pi .
\end{equation}
The geometry is spherical for
\begin{equation}
\frac{3}{2} \pi <\beta < 3 \pi ;
\end{equation}
Nil geometry for $\beta =\frac{3}{2}\pi $; and  $\widetilde{SL(2,\mathbb{R}})$ geometry for $\beta < \frac{3}{2}\pi $.

The  conemanifold structure $\left(O(m),\beta \right)$, whose singular set is  the $(2,1)$-fibre of the above Seifert structure, have the following limits of sphericity
\begin{equation}
   \beta _{L}= \frac{5}{3}\pi , \quad  \beta _{U}= \frac{11}{3}\pi .
\end{equation}
The geometry is spherical for
\begin{equation}
\frac{5}{3} \pi <\beta < \frac{11}{3} \pi ;
\end{equation}
Nil geometry for $\beta =\frac{5}{3}\pi $; and  $\widetilde{SL(2,\mathbb{R}})$ geometry for $\beta < \frac{5}{3}\pi $.

\subsection{$\{a_1,a_2,a_3\}=\{2,3,5\}$}\hspace*{\fill} \\

The manifold
\[
M=(O,o,0\, |\, b;(2,1),(3,b_2),(5,b_{3}))=\left\langle O,o,0\, |\, -1;(2,1),(3,b_2),(5,(b+1)5+b_3)\right\rangle
\]
where $b_{2}$ is 1 or 2, and $b_{3}$ is 1, 2, 3 or 4,  and
\[
-e= \frac{30b+15+10b_2+6b_3}{30}  \neq 0 \quad \Rightarrow \quad m:=30b+15+10b_2+6b_3\neq 0.
\]
is  denoted  by $I(m)$. Observe that the integer $m$, determines $b,\, b_2,\, b_3$.

The simplest case is $I(1)=(O,o,0\, |\, -1;(2,1),(3,1),(5,1))$, whose fundamental group is the binary icosahedral group $I^*$ of order 120, and it is the manifold of spherical tessellations $M(S532)$ (\cite{M1987}).

The Seifert manifold $I(m)$ has geometric conemanifold  structures, whose singular set is composed by three or two exceptional fibres, according to Theorem \ref{t3fibras} and Theorem \ref{t2fibras}. The three geometries (Spherical, Nil and $\widetilde{SL(2,\mathbb{R}})$) are possible.

The  conemanifold structures $\left(I(m),\beta \right)$, whose singular set is  the $(5,b_3)$-fibre of the above Seifert structure,  have the following limits of sphericity
\begin{equation}
   \beta _{L}= \frac{5}{3}\pi , \quad  \beta _{U}= \frac{25}{3}\pi .
\end{equation}
The geometry is spherical for
\begin{equation}
\frac{5}{3} \pi <\beta < \frac{25}{3} \pi ;
\end{equation}
Nil geometry for $\beta =\frac{5}{3}\pi $; and  $\widetilde{SL(2,\mathbb{R}})$ geometry for $\beta < \frac{5}{3}\pi $.

The  conemanifold structures $\left(I(m),\beta \right)$, whose singular set is  the $(3,b_2)$-fibre of the above Seifert structure, have the following limits of sphericity
\begin{equation}
   \beta _{L}= \frac{9}{5}\pi , \quad  \beta _{U}= \frac{21}{5}\pi .
\end{equation}
The geometry is spherical for
\begin{equation}
\frac{9}{5} \pi <\beta < \frac{21}{5} \pi ;
\end{equation}
Nil geometry for $\beta =\frac{9}{5}\pi $; and  $\widetilde{SL(2,\mathbb{R}})$ geometry for $\beta < \frac{9}{5}\pi $.

The  conemanifold structures $\left(I(m),\beta \right)$, whose singular set is  the $(2,1)$-fibre of the above Seifert structure, have the following limits of sphericity
\begin{equation}
   \beta _{L}= \frac{28}{15}\pi , \quad  \beta _{U}= \frac{52}{15}\pi .
\end{equation}
The geometry is spherical for
\begin{equation}
\frac{28}{15} \pi <\beta < \frac{52}{15} \pi ;
\end{equation}
Nil geometry for $\beta =\frac{28}{15}\pi $; and  $\widetilde{SL(2,\mathbb{R}})$ geometry for $\beta < \frac{28}{15}\pi $.

\section{$\mathcal{F}_2$: Seifert manifolds with Nil or Euclidean geometry}\label{s5}

We now study the Seifert manifolds with orbit space $S^{2}$  and with three exceptional fibres possessing  a Nil or Euclidean geometry. Because the orbit space is a Euclidean orbifold, the multiplicities of the exceptional fibres are $(3,3,3),\, (2,4,4)$ and $(2,3,6)$.

All of them have geometric conemanifold structures with three or two exceptional fibres as singular set according to Theorem \ref{t3fibras} and Theorem \ref{t2fibras}. The three geometries, Spherical, Nil and $\widetilde{SL(2,\mathbb{R}})$ ($S^2\times \mathbb{R}$, Euclidean and $H^2\times \mathbb{R}$), when the Euler class $e\neq 0$ ($e=0$), are possible.

 Next we collect the results for the conemanifold structures on these manifolds whose singular set is just one exceptional fibre.
Observe that in all cases $\beta_{L}=2\pi$, and therefore the corresponding spherical ($S^2\times \mathbb{R}$) conemanifold structures, for $e\neq 0$ ($e=0$), have angle $\beta>2\pi$.

\subsection{$\{a_1,a_2,a_3\}=\{3,3,3\}$}\hspace*{\fill} \\

The manifold
\[
(O,o,0\, |\, b;(3,b_1),(3,b_2),(3,b_{3}))=\left\langle O,o,0\, |\, -1;(3,b_1),(3,b_2),(3,(b+1)3+b_3)\right\rangle
\]
where $b_{i}$ is 1 or 2,  and
\[
e=- \frac{3b+b_1+b_2+b_3}{3}   \quad \Rightarrow \quad m:=3b+b_1+b_2+b_3.
\]
is  denoted  by $N_{(3,3,3)}(m,n)$, where $n=\min (b_1,b_2,b_3)$. Observe that the integer $m$ and $n\in \{1,2\}$, determine $\{ b,\, b_1\, b_2,\, b_3\}$. In fact when $m$ is a multiple of 3, $m=3k$, two solutions for $\{ b, b_1,b_2,b_3 \}$ are posible: $\{ k-1, 1,1,1 \}$ and $\{ k-2, 2,2,2 \}$. For $m=3k+1$ the solution is $\{ k-1, 1,1,2 \}$, and for $m=3k+2$ the solution is $\{ k-1, 1,2,2 \}$.

The Euclidean cases correspond to $e=0$ $\Rightarrow$ $m=0$. They  are
\[
N_{(3,3,3)}(0,1)=(O,o,0\, |\, -1;(3,1),(3,1),(3,1))\]
\[
N_{(3,3,3)}(0,2)=(O,o,0\, |\, -2;(3,2),(3,2),(3,2))
\]

The first one, $N_{(3,3,3)}(0,1)$,  is the manifold of Euclidean tessellations $M(S333)$ (\cite{M1987}).

The  conemanifold structures $\left(N_{(3,3,3)}(m,n),\beta \right)$, whose singular set is an exceptional fibre of the above Seifert structure,  have the following limits of sphericity
\begin{equation}
   \beta _{L}= 2\pi , \quad  \beta _{U}= 4\pi .
\end{equation}
The geometry is spherical ($S^2\times \mathbb{R}$) for
\begin{equation}
2 \pi <\beta < 4 \pi ;
\end{equation}
Nil (Euclidean) geometry for $\beta =2\pi $; and  $\widetilde{SL(2,\mathbb{R}})$ ($H^2\times \mathbb{R}$) geometry for $\beta < 2\pi $.

\subsection{$\{a_1,a_2,a_3\}=\{2,4,4\}$}\hspace*{\fill} \\

The manifold
\[
(O,o,0\, |\, b;(2,1),(4,b_2),(4,b_{3}))=\left\langle O,o,0\, |\, -1;(2,1),(4,b_2),(4,(b+1)4+b_3)\right\rangle
\]
where $b_{i}$ is 1 or 3,  and
\[
e=- \frac{4b+2+b_2+b_3}{4}   \quad \Rightarrow \quad m:=4b+2+b_2+b_3.
\]
is  denoted by $N_{(2,4,4)}(m,n)$, where $n=\min (b_2,b_3)$. Observe that the even integer $m$ and $n\in \{1,3\}$, determine $b,\, b_2,\, b_3$. In fact when $m$ is a multiple of 4, $m=4k$, two solutions for $\{ b, b_2,b_3 \}$ are posible: $\{ k-1, 1,1 \}$ and $\{ k-2, 3,3 \}$. For $m=4k+2$ the solution is $\{ k-1, 1,3 \}$.

The Euclidean cases correspond to $e=0$ $\Rightarrow$ $m=0$. They  are
\[
N_{(2,4,4)}(0,1)=(O,o,0\, |\, -1;(2,1),(4,1),(4,1))\]
\[
 N_{(2,4,4)}(0,3)=(O,o,0\, |\, -2;(2,1),(4,3),(4,3))
\]
The first one, $N_{(2,4,4)}(0,1)$,  is the manifold of Euclidean tessellations $M(S244)$ (\cite{M1987}).

The  conemanifold structures $\left(N_{(2,4,4)}(m),\beta \right)$, whose singular set is the $(2,1)$-fibre  of the above Seifert structure, have the following limits of sphericity
\begin{equation}
   \beta _{L}= 2\pi , \quad  \beta _{U}= 4\pi .
\end{equation}
The geometry is spherical ($S^2\times \mathbb{R}$) for
\begin{equation}
2 \pi <\beta < 4 \pi ;
\end{equation}
Nil (Euclidean) geometry for $\beta =2\pi $; and  $\widetilde{SL(2,\mathbb{R}})$ ($H^2\times \mathbb{R}$)geometry for $\beta < 2\pi $.

The  conemanifold structures $\left(N_{(2,4,4)}(m,n),\beta \right)$, whose singular set is a $(4,b_i)$-fibre of the above Seifert structure,  have the following limits of sphericity
\begin{equation}
   \beta _{L}= 2\pi , \quad  \beta _{U}= 3\pi .
\end{equation}
The geometry is spherical ($S^2\times \mathbb{R}$) for
\begin{equation}
2 \pi <\beta < 3 \pi ;
\end{equation}
Nil (Euclidean) geometry for $\beta =2\pi $; and  $\widetilde{SL(2,\mathbb{R}})$ ($H^2\times \mathbb{R}$)geometry for $\beta < 2\pi $.

\subsection{$\{a_1,a_2,a_3\}=\{2,3,6\}$}\hspace*{\fill} \\

The manifold
\[
(O,o,0\, |\, b;(2,1),(3,b_2),(6,b_{3}))=\left\langle O,o,0\, |\, -1;(2,1),(3,b_2),(6,(b+1)6+b_3)\right\rangle
\]
where $b_{2}$ is 1 or 2,  $b_{3}$ is 1 or 5,  and
\[
e= -\frac{6b+3+2b_2+b_3}{6}   \quad \Rightarrow \quad m:=6b+3+2b_2+b_3.
\]
is  denoted  by $N_{(2,3,6)}(m,n)$, where $n=\min (b_2,b_3)$. Observe that the integer $m$ which is a even integer, and $n\in \{1,2\}$ determine $b,\, b_2,\, b_3$. In fact when $m$ is a multiple of 6, $m=6k$, two solutions for $\{ b, b_2,b_3 \}$ are posible: $\{ k-1, 1,1 \}$ and $\{ k-2, 2,5 \}$. For $m=6k+2$ the solution is $\{ k-1, 2,1 \}$, and for $m=6k+4$ the solution is $\{ k-1, 1,5 \}$.

\[
N_{(2,3,6)}(0,1)=(O,o,0\, |\, -1;(2,1),(3,1),(6,1))\]
\[
 N_{(2,3,6)}(0,2)=(O,o,0\, |\, -2;(2,1),(3,2),(6,5))
\]
The first one, $N_{(2,4,4)}(0,1)$,  is the manifold of Euclidean tessellations $M(S236)$ (\cite{M1987}).

The  conemanifold structures $\left(N_{(2,3,6)}(m),\beta \right)$, whose singular set is the $(2,1)$-fibre  of the above Seifert structure,  have the following limits of sphericity
\begin{equation}
   \beta _{L}= 2\pi , \quad  \beta _{U}= \frac{8}{3}\pi .
\end{equation}
The geometry is spherical ($S^2\times \mathbb{R}$) for
\begin{equation}
2 \pi <\beta < \frac{8}{3} \pi ;
\end{equation}
Nil (Euclidean) geometry for $\beta =2\pi $; and  $\widetilde{SL(2,\mathbb{R}})$ ($H^2\times \mathbb{R}$) geometry for $\beta < 2\pi $.

The  conemanifold structures $\left(N(m)_{(2,3,6)},\beta \right)$, whose singular set is the $(3,b_2)$-fibre  of the above Seifert structure, have the following limits of sphericity
\begin{equation}
   \beta _{L}= 2\pi , \quad  \beta _{U}= 4\pi .
\end{equation}
The geometry is spherical ($S^2\times \mathbb{R}$) for
\begin{equation}
2 \pi <\beta < 4 \pi ;
\end{equation}
Nil (Euclidean) geometry for $\beta =2\pi $; and  $\widetilde{SL(2,\mathbb{R}})$ ($H^2\times \mathbb{R}$) geometry for $\beta < 2\pi $.

The  conemanifold structures $\left(N(m)_{(2,3,6)},\beta \right)$, whose singular set is the $(6,b_3)$-fibre of the above Seifert structure,  have the following limits of sphericity
\begin{equation}
   \beta _{L}= 2\pi , \quad  \beta _{U}= 10\pi .
\end{equation}
The geometry is spherical ($S^2\times \mathbb{R}$) for
\begin{equation}
2 \pi <\beta < 10 \pi ;
\end{equation}
Nil (Euclidean) geometry for $\beta =2\pi $; and  $\widetilde{SL(2,\mathbb{R}})$ ($H^2\times \mathbb{R}$) geometry for $\beta < 2\pi $.

\section{$\mathcal{F}_3$: Manifolds obtained by Dehn surgery along the torus knot $K_{(r,s)}$}\label{s6}

The torus knot $K_{(r,s)}$ is a general fibre of the Seifert manifold structure in $S^{3}$
\[
S_{(r,s)}= (O,o,0\, |\, -1;(s,b_1),(r,b_2))
\]
where $(s,b_1)$, $(r,b_2)$ and $(r,s)$ are pairs of coprime integer numbers, $0<b_1<s$, $0<b_2<r$ and $|-rs+b_1r+b_2s|=1$. We suppose  that $r>s>1$. Actually there are two possible values for the pair $(b_1 ,b_2)$ with the above conditions. They  correspond to the two possible different orientations on $S^{3}$, defining the torus knot $K_{(r,s)}$ and its mirror image $K^{*}_{(r,s)}$. Concretely, for $-rs+b_1r+b_2s=1$ the general fibre is the right handle torus knot, and for $-rs+b_1r+b_2s=-1$ the general fibre is the left handle torus knot.

The classification of the manifolds obtained by $(p/q)$-Dehn surgery along a torus knot is given in \cite{M1971}. Recall the concept of $(p/q)$-Dehn surgery along a knot $K\subset S^{3}$: Consider a homology basis $(\overrightarrow{M},\overrightarrow{L})$ of the boundary of a regular neighborhood $N$ of the oriented knot $\overrightarrow{K}$, where $\overrightarrow{M}$ is an oriented meridian and $\overrightarrow{L}$ is an oriented canonical longitude homologous to $\overrightarrow{K}$ in $N$, such that the linking number of $\overrightarrow{M}$ and $\overrightarrow{K}$ is +1 in $\overrightarrow{S^3}$.  The result of $(p/q)$-Dehn surgery is the manifold $\overline{S^{3}\setminus N}\bigcup_{h} (D^{2}\times S^{1})$ where the homeomorphism $h:\partial(D^{2}\times S^{1})\longrightarrow \partial N$ is defined by $h(\partial D^{2})=p\overrightarrow{M}+q\overrightarrow{L}$. We suppose $p>0$ because $(p/q)=((-p)/(-q))$.

If $qrs+p\neq 0$, the result of $(p/q)$-Dehn surgery in a left handle torus knot $K_{(r,s)}$ is the Seifert manifold
\[
S_{(r,s)}(p/q)= \left\langle O,o,0\, |\, -1;(s,b_1),(r,b_2),(|qrs+p|,\varepsilon q) \right\rangle
\]
where $\varepsilon =1$ if $qrs+p>0$ and $\varepsilon =-1$ if $qrs+p<0$, and $-rs+b_1r+b_2s=-1$.

If $-qrs+p\neq 0$, the result of $(p/q)$ Dehn surgery in a right handle torus knot $K^{*}_{(r,s)}$ is the Seifert manifold
\[
S'_{(r,s)}(p/q)= \left\langle O,o,0\, |\, -1;(s,b_1),(r,b_2),(|-qrs+p|,\varepsilon q) \right\rangle
\]
where $\varepsilon =1$ if $-qrs+p>0$ and $\varepsilon =-1$ if $-qrs+p<0$, and $-rs+b_1r+b_2s=1$.

The reason of that difference is that the canonical longitud $\overrightarrow{L}$ is related to the toroidal longitud $\overrightarrow{l_{t}}$ (the longitude which is a fibre of the Seifert fibration) by the formula $\overrightarrow{l_{t}}= \overrightarrow{L}-rs\overrightarrow{M}$ for the left handle torus and $\overrightarrow{l_{t}}= \overrightarrow{L}+rs\overrightarrow{M}$ for the right handle case.

The information about the geometric conemanifold structures on the Seifert manifolds obtained by Dehn surgery along a left (right) handle torus knot $K_{(r,s)}$ ($K^{*}_{(r,s)}$) can be collected in a 2-positive coordinate graph. We develop here the case of the left handle torus knots.  The points in the graph bearing the symbol $\blacklozenge$ correspond to points $(m,n)$ such that $m,n\in \mathbb{Z}$, $\gcd (m,n)=1$. The point $(x,y)$, such that $x/y$ is a rational number ($x/y=m/n$, $m,n\in \mathbb{Z}$, $\gcd (m,n)=1$) represents the conemanifold $\left(S_{(r,s)}(p/q),\beta =2\pi m/x \right)$, where
\[
m=|p+qrs|,\quad  n=\varepsilon q, \quad \Rightarrow \quad \frac{p}{q}=\frac{m-rsn}{n}.
 \]
 The points in the straight line through the origin $l_{m/n}$ with rational slope $m/n$ parametrize the geometric conemanifold structures in the Seifert manifold $S_{(r,s)}((m-rsn)/n)$ when $x > \frac{rs}{rs-r+s}$, by Theorem \ref{t1fibra}. The right handle case is similar but the line $l_{m/n}$ represents the oriented manifold $S'_{(r,s)}((m+rsn)/n)$ resulting from  $((m+rsn)/n)$-Dehn surgery on the right handle torus knot $K^{*}_{(r,s)}$.

  Actually, it follows from (\ref{ebetaalfa}) and from the expressions  $m=|p+qrs|$ and $\beta =2\pi m/x$, that the $x$ coordinate is equal to $\pi /\alpha$, that is, $\alpha= \pi /x$. Hence in the graphs the vertical lines represents angles.

 Consider the manifold $S_{(r,s)}((m-rsn)/n)$. Its conemanifold structure corresponds to the line $l_{m/n}$. The geometric structures supported by this conemanifold structure correspond to the points in this line. If the Euler class of the manifold $S_{(r,s)}((m-rsn)/n)$ is different from zero, the possible geometries for the conemanifold structure are spherical, Nil or $\widetilde{SL(2,\mathbb{R}})$. The geometry is spherical for
\[
\frac{rs}{rs-r+s}<x<\frac{rs}{rs-r-s};
 \]
 Nil for $x=\frac{rs}{rs-r-s}$; and $\widetilde{SL(2,\mathbb{R})}$ for
 \[
 x>\frac{rs}{rs-r-s}.
  \]
 And the intersection of the vertical lines $\mathcal{U}:x=\frac{rs}{rs-r+s}$ and $\mathcal{L}:x=\frac{rs}{rs-r-s}$  with $l_{m/n}$  are the limits of sphericity $\beta _{U}$ and $\beta_{L}$ respectively of these geometric conemanifold  structures.

 If the Euler class of the manifold $S_{(r,s)}((m-rsn)/n)$ is zero, the possible geometries for the conemanifold structure are $S^2\times \mathbb{R}$, Euclidean or $H^2\times \mathbb{R}$. The geometry is $S^2 \times \mathbb{R}$ for
\[
\frac{rs}{rs-r+s}<x<\frac{rs}{rs-r-s};
 \]
 Euclidean for $x=\frac{rs}{rs-r-s}$; and $H^2 \times \mathbb{R}$ for
 \[
 x>\frac{rs}{rs-r-s}.
  \]The following lema identify these manifolds.

 \begin{lemma}
 The $(0/1)$-Dehn surgery in the left (right) handle torus knot $K_{(r,s)}$ ($K^{*}_{(r,s)}$) yields Seifert manifold with $S^2\times \mathbb{R}$, Euclidean or $H^2\times \mathbb{R}$ geometry.
 \end{lemma}

 \begin{proof}
 The  Seifert manifold
 \[
S_{(r,s)}((m-rsn)/n)= \left\langle O,o,0\, |\, -1;(s,b_1),(r,b_2),(m,n) \right\rangle
\]
 is the result of $((m-rsn)/n)$-Dehn surgery in the left handle torus knot $K_{(r,s)}$, where $-rs+b_1r+b_2s=-1$. Then
 \[
e=1-\frac{b_1}{s}-\frac{b_2}{r}-\frac{n}{m}=\frac{rs-b_2s-b_1r}{rs}-\frac{n}{m}=\frac{1}{rs}-\frac{n}{m}=\frac{m-nrs}{rsm}
\]
Therefore $e=0$ when $m-rsn=0$. This is $(0/1)$-Dehn surgery in the left handle torus knot $K_{(r,s)}$.

Analogously, the Seifert manifold
\[
S'_{(r,s)}((m+rsn)/n)=\left\langle O,o,0\, |\, -1;(s,b_1),(r,b_2),(m,n) \right\rangle
\]
 is the result of $((m+rsn)/n)$-Dehn surgery on the right handle torus knot $K^{*}_{(r,s)}$, where $-rs+b_1r+b_2s=1$. Here
  \[
e=1-\frac{b_1}{s}-\frac{b_2}{r}-\frac{n}{m}=\frac{rs-b_2s-b_1r}{rs}-\frac{n}{m}=\frac{-1}{rs}-\frac{n}{m}=\frac{-m-nrs}{rsm}
\]
Therefore $e=0$ when $-m-rsn=0$. This is $(0/1)$-Dehn surgery in the left handle torus knot $K^{*}_{(r,s)}$.
 \end{proof}

\begin{corollary}
There exists only one line, $l_{m/n}$, in the graph associated a torus knot, whose points  correspond to conemanifold structures with $S^2\times \mathbb{R}$, Euclidean or $H^2\times \mathbb{R}$ geometry. For the left handle torus knot $K_{(r,s)}$ the line is $l_{rs/1}$, and for the right handle torus knot $K^{*}_{(r,s)}$ the line is $l_{rs/(-)1}$.
\end{corollary}

 \begin{figure}[h]
\begin{center}\epsfig{file=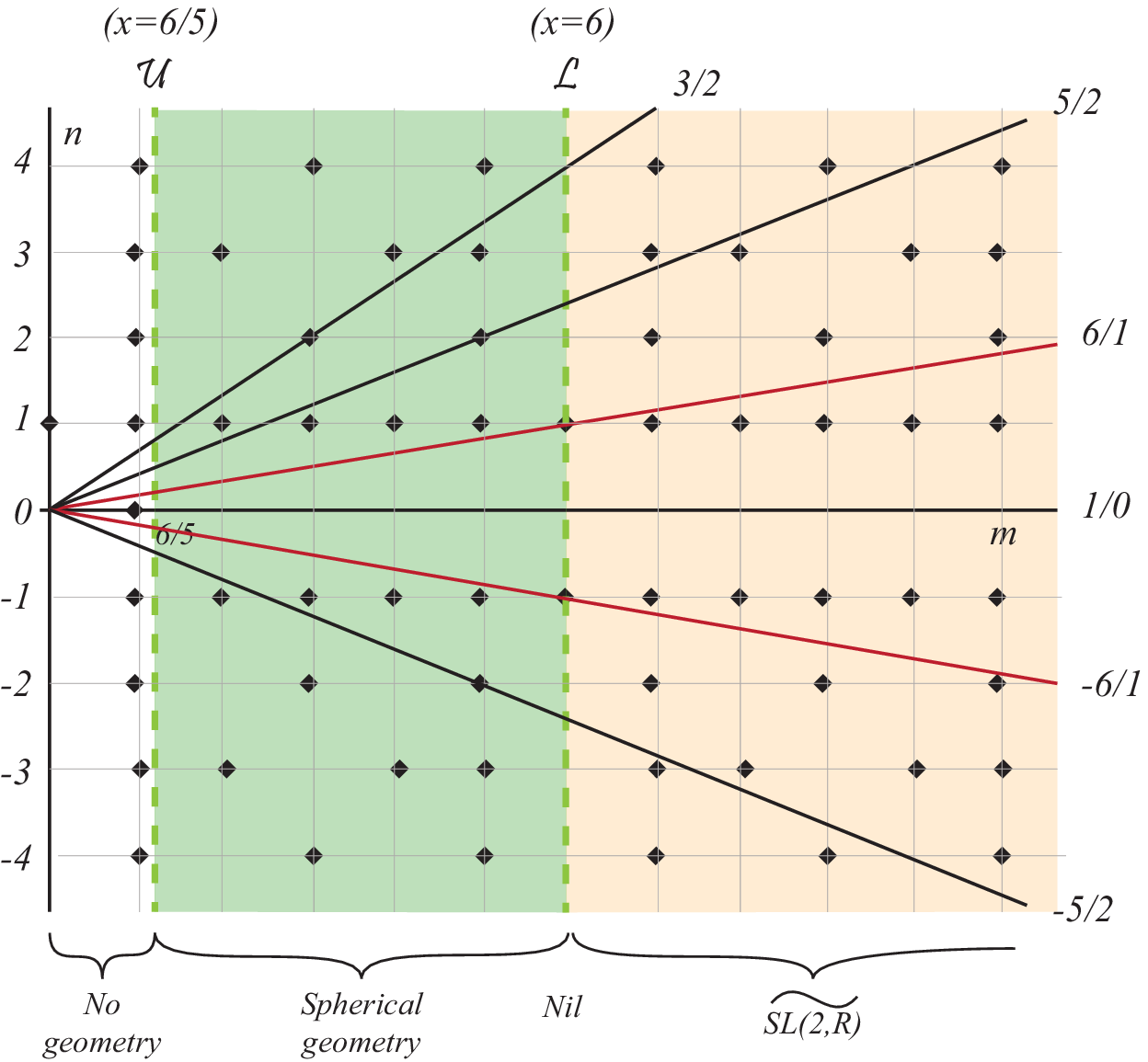,height=6cm}
\caption{The graph for the trefoil knots $K_{(3,2)}$ and $K^{*}_{(3,2)}$}\label{fplot32}\end{center}
\end{figure}

Figure \ref{fplot32} shows the graph for the trefoil knot $K_{(3,2)}$. The left handle trefoil was studied in \cite{LM2015},  where the following affirmations were deduced.
\begin{enumerate}
 \item[(1)] For all $S_{(3,2)}\left( p/q\right)$ the upper sphericity limit $\beta_{U}$ is equal to five times the lower sphericity limit $\beta_{L}$
 \item[(2)] If $\left(S_{(3,2)}\left( p/q\right),z\right)$ has a spherical orbifold structure then $2\leq z\leq 5$.
\item[(3)] There exist infinitely many  Nil orbifold structures:
 \begin{enumerate}
   \item $\left( S_{(3,2)}\left( (2-6y)/y\right),3\right)$, where $\gcd (2,y)=1$.
   \item $\left(S_{(3,2)}\left( (3-6y)/y\right),2\right)$, where $\gcd (3,y)=1$.
   \item $S_{(3,2)}\left( (6-6y)/y\right)$, where $\gcd (6,y)=1$. These are manifolds. Hence
there are infinitely many non-singular Nil manifold structures.
 \item[(4)] If $\left(S_{(3,2)}\left( p/q\right),z\right)$  has a $\widetilde{SL(2,R)}$  orbifold structure, then $z>6$.
 \end{enumerate}
\end{enumerate}

The same graph is associated to the right handle trefoil $K^{*}_{(3,2)}$. The above conclusions (1) (2) and (4) still hold true for the manifold  $S'_{(3,2)}\left( p/q\right)$. But conclusion 3 must be replaced by 3'
\begin{enumerate}
  \item[(3)'] : There exist infinitely many  Nil orbifold structures:
 \begin{enumerate}
   \item $\left( S'_{(3,2)}\left( (2+6y)/y\right),3\right)$, where $\gcd (2,y)=1$.
   \item $\left(S'_{(3,2)}\left( (3+6y)/y\right),2\right)$, where $\gcd (3,y)=1$.
   \item $S'_{(3,2)}\left( (6+6y)/y\right)$, where $\gcd (6,y)=1$. These are manifolds. Hence
   there exist infinitely many non-singular Nil manifold structures.
 \end{enumerate}
\end{enumerate}

We now add the following results:
\begin{remark}
The Seifert manifolds represented by the two lines $l_{6/1}$ and $l_{6/(-1)}$,
\begin{eqnarray*}
 S_{(3,2)}\left( 0/1\right) &=&(O,o,0\, |\, -1;(2,1),(3,1),(6,1)) \\
 S'_{(3,2)}\left( 0/1\right) &=& \left\langle O,o,0\, |\, -1;(2,1),(3,2),(6,-1)\right\rangle ,
\end{eqnarray*}
do not have $S^2\times\mathbb{R}$ orbifold structure with the core of the surgery as singular set. They have  Euclidean manifold structure and all the orbifols structures have $H^2\times\mathbb{R}$ geometry.
\end{remark}

Next, we analize the case of a general torus knot.

To study the existence of spherical orbifolds among the conemanifold structures in $S_{(r,s)}\left( p/q\right)$ we have to study the points with integer coordinates  lying between  the two straight lines $\mathcal{U}:x=\frac{rs}{rs-r+s}$ and $\mathcal{L}:x=\frac{rs}{rs-r-s}$. The Nil geometric orbifolds correspond to integer coordinates in the straight line
 $\mathcal{L}:x=\frac{rs}{rs-r-s}$. Given two integer numbers $r>s>1$ let's define
 \[
 x_{U}=\frac{rs}{rs-r+s} \quad x_{L}=\frac{rs}{rs-r-s}
 \]

\begin{lemma}\label{lxl}
Let $r>s>1$. The rational number $x_{L}$ is less than 2, exactly for all pairs $(r,s)$, with $r\geq 6$ and $s\geq 3$, and for the pair $(5,4)$.
\end{lemma}

\begin{proof}
Observe that
\begin{gather}\label{exl}
x_{L}=\frac{rs}{rs-r-s}<2 \quad \Leftrightarrow \quad rs<2(rs-r-s) \quad \Leftrightarrow \notag \\ \Leftrightarrow \quad 0<rs-2r-2s=(2+r-2)(s-2)-2s=(r-2)(s-2)+2s-2s-4 \quad \Leftrightarrow \notag\\
\Leftrightarrow \quad 4< (r-2)(s-2)
\end{gather}
This is true if $r\geq 6$ and $s\geq 3$. It is false for $s=2$ and any $r$. For the remaining cases $(4,3)$, $(5,3)$  and $(5,4)$, it is  only true for $(5,4)$.
\end{proof}

\begin{lemma}\label{lxl2}
Let $r>2$. The rational number $x_{L}$ for the pairs $(r,2)$ is bigger than 3, exactly for $r=3,4,5$.
\end{lemma}

\begin{proof}
Observe that
\[
\frac{2r}{2r-r-2}=\frac{2r}{r-2}>3\quad \Leftrightarrow \quad 2r>3r-6 \quad \Leftrightarrow \quad r<6
\]
\end{proof}

\begin{lemma}
The Seifert manifold
\[
\left\langle O,o,0\, |\,  -1;(s,b_1),(r,b_2),(m,n) \right\rangle
\]
where $(s,b_1)$, $(r,b_2)$ and $(r,s)$ are pairs of coprime integer numbers, $0<b_1<s$, $0<b_2<r$, $s<r$, and $|-rs+b_1r+b_2s|=1$, is the result of Dehn surgery on the torus knot $K_{(r,s)}$ or $K^{*}_{(r,s)}$.
\end{lemma}

\begin{proof}
If $-rs+b_1r+b_2s=-1$, then the manifold is the result of $\left( (m-rsn)/n\right)$ Dehn surgery on the left handle torus knot $K_{(r,s)}$: $S_{(r,s)}\left( (m-rsn)/n\right)$.

If $-rs+b_1r+b_2s=1$, then the manifold is the result of $\left( (m+rsn)/n\right)$ Dehn surgery on the right handle torus knot $K^{*}_{(r,s)}$: $S'_{(r,s)}\left( (m+rsn)/n\right)$.
\end{proof}

\begin{theorem}
The only Seifert manifolds
\[
\left\langle O,o,0\, |\,  -1;(s,b_1),(r,b_2),(m,n) \right\rangle
\]
where $(s,b_1)$, $(r,b_2)$ and $(r,s)$ are pairs of coprime integer numbers, $0<b_1<s$, $0<b_2<r$, $s<r$, $e\neq 0$ and $|-rs+b_1r+b_2s|=1$, supporting spherical orbifold structures with singular set the exceptional $(m,n)$-fibre, are those obtained by Dehn surgery in the torus knots $(r,2)$, for all $r>2$, $(4,3)$, and $(5,3)$.

The angle around the singular geodesic is $\pi$, for $(4,3)$, $(5,3)$ and $(r,2)$ $r>6$; $\pi$ or $2\pi /3$, for $(5,2)$; and  $\pi$, $2\pi /3$, $\pi /2$ or $2\pi /5$, for $(3,2)$.
\end{theorem}

\begin{proof}
The points in the graph corresponding to spherical orbifold structures are points with integer coordinates lying between the lines $\mathcal{U}$ and $\mathcal{L}$. Observe that always $x_{U}>1$. Lemma \ref{lxl} gives the pairs where $x_{L}<2$. Therefore, for all those cases, there are no points with integer coordinates in the spherical region. For the other cases there are always points with coordinates $(2,n)$ corresponding to spherical orbifold structures with angle $\pi$ along the core of the surgery if $n$ is even, and non singular spherical structure if $n$ is odd.

On the other hand, the only cases with $x_{L}>3$  are $(r,s)$ with $6>r>s>1$, whose graphs are depicted in Figures \ref{fplot32}, \ref{fplot43} and \ref{fplots}.
\end{proof}

\begin{corollary}
The only Seifert manifolds
\[
S_{(r,s)}\left( p/q\right)=\left\langle O,o,0\, |\, -1;(s,b_1),(r,b_2),(qrs+p,q)\right\rangle
\]
where $(s,b_1)$, $(r,b_2)$ and $(r,s)$ are pairs of coprime integer numbers, $0<b_1<s$, $0<b_2<r$, $s<r$, $e\neq 0$ and $|-rs+b_1r+b_2s|=1$, supporting Nil orbifold or manifold structures with singular set the exceptional $(qrs+p,q)$-fibre, are those obtained by Dehn surgery in the trefoil knots $K_{(3,2)}$ or $K^{*}_{(3,2)}$
\end{corollary}

\begin{proof}
The only value of $(r,s)$, $r>s>1$ for which $x_{L}$ is an integer is $(3,2)$.
\end{proof}

 \begin{figure}[h]
\begin{center}\epsfig{file=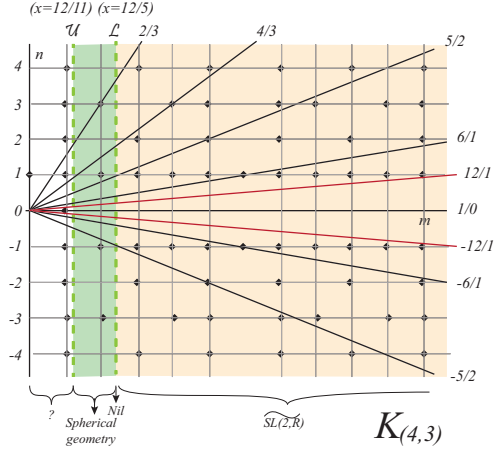,height=6cm}
\caption{The graph for the knots $K_{(4,3)}$ and $K^{*}_{(4,3)}$} \label{fplot43}\end{center}
\end{figure}

\subsection{Conclusions for the knots $K_{(4,3)}$ and $K^{*}_{(4,3)}$}\hspace*{\fill} \\

\begin{enumerate}
\item Points in the line $l_{12/1}$ ($l_{12/(-1)}$) represent the geometric conemanifold structures in the manifold  $S_{(4,3)}\left( 0/1\right)$ ($S'_{(4,3)}\left( 0/1\right)$). The possible geometries in both cases are $(S^2\times\mathbb{R})$, Euclidean and $(H^2\times\mathbb{R})$. These conemanifolds structures are included in the next conclusions where we write in brackets the corresponding geometry.
\item For all $S_{(4,3)}\left( p/q\right)$ and $S'_{(4,3)}\left( p/q\right)$ the upper sphericity limit $\beta_{U}$ is equal to $11/5$ times the lower sphericity limit $\beta_{L}$.
  \item There are no Nil (Euclidean) manifold structures on $S_{(4,3)}\left( p/q\right)$ and \newline $S'_{(4,3)}\left( p/q\right)$.
  \item There are no  Nil (Euclidean) orbifold structures on $S_{(4,3)}\left( p/q\right)$ and $S'_{(4,3)}\left( p/q\right)$.
  \item The orbifold $\left(S_{(4,3)}\left( (12v-1)/(-v)\right),2\right)$, $v\in \mathbb{N}$ has a spherical orbifold structure with the core of the surgery as singular set labeled 2. It is the Seifert manifold
      \[
      \left\langle O,o,0\, |\, -1;(3,1),(4,3),(1,v)\right\rangle =(O,o,0\, |\, v-1;(3,1),(4,3))
       \]
       which is the lens space $L(12v+1,16v)$ studied in Subsection 4.2.
\item The orbifold $\left(S'_{(4,3)}\left( (1+12v)/v\right),2\right)$, $v\in \mathbb{N}$ has a spherical orbifold structure with the core of the surgery as singular set labeled 2. It is the Seifert manifold
    \[
    \left\langle O,o,0\, |\, -1;(3,2),(4,1),(1,v)\right\rangle =(O,o,0\, |\, v-1;(3,2),(4,1)).
     \]
     This is the lens space $L(12v-1,-2+15v)$ studied in Subsection 4.2.
  \item The manifold $S_{(4,3)}\left( (12v-2)/(-v)\right)$, $v$ odd integer, which is the Seifert manifold
  \[
  \left\langle O,o,0\, |\, -1;(3,1),(4,3),(2,v)\right\rangle =(O,o,0\, |\, v-2;(3,1),(4,3),(2,1))
   \]has a spherical manifold structure with the core of the surgery as geodesic. This is the spherical manifold $O(12v-5)$ studied in Subsection 4.5.
\item The manifold $S'_{(4,3)}\left( (2+12v)/v\right)$, $v$ odd integer, which is the Seifert manifold
\[
\left\langle O,o,0\, |\, -1;(3,2),(4,1),(2,v)\right\rangle =(O,o,0\, |\, v-2;(3,2),(4,1),(2,1))
 \]
 has a spherical manifold structure with the core of the surgery as geodesic. This is the spherical manifold $O(12v-7)$ studied in Subsection 4.5.
  \item All the other  manifold or orbifold structures in $S_{(4,3)}\left( p/q\right)$ and $S'_{(4,3)}\left( p/q\right)$ with  the core of the surgery as geodesic or singular set, have $\widetilde{SL(2,\mathbb{R})}$ ($H^2\times\mathbb{R}$) geometry.
\end{enumerate}

\subsection{Conclusions for the knots $K_{(5,s)}$ and $K^{*}_{(5,s)}$, $s=2,3,4$}\hspace*{\fill} \\

Figure \ref{fplots} shows the graph for the cases $r=5$.
 \begin{figure}[h]
\begin{center}\epsfig{file=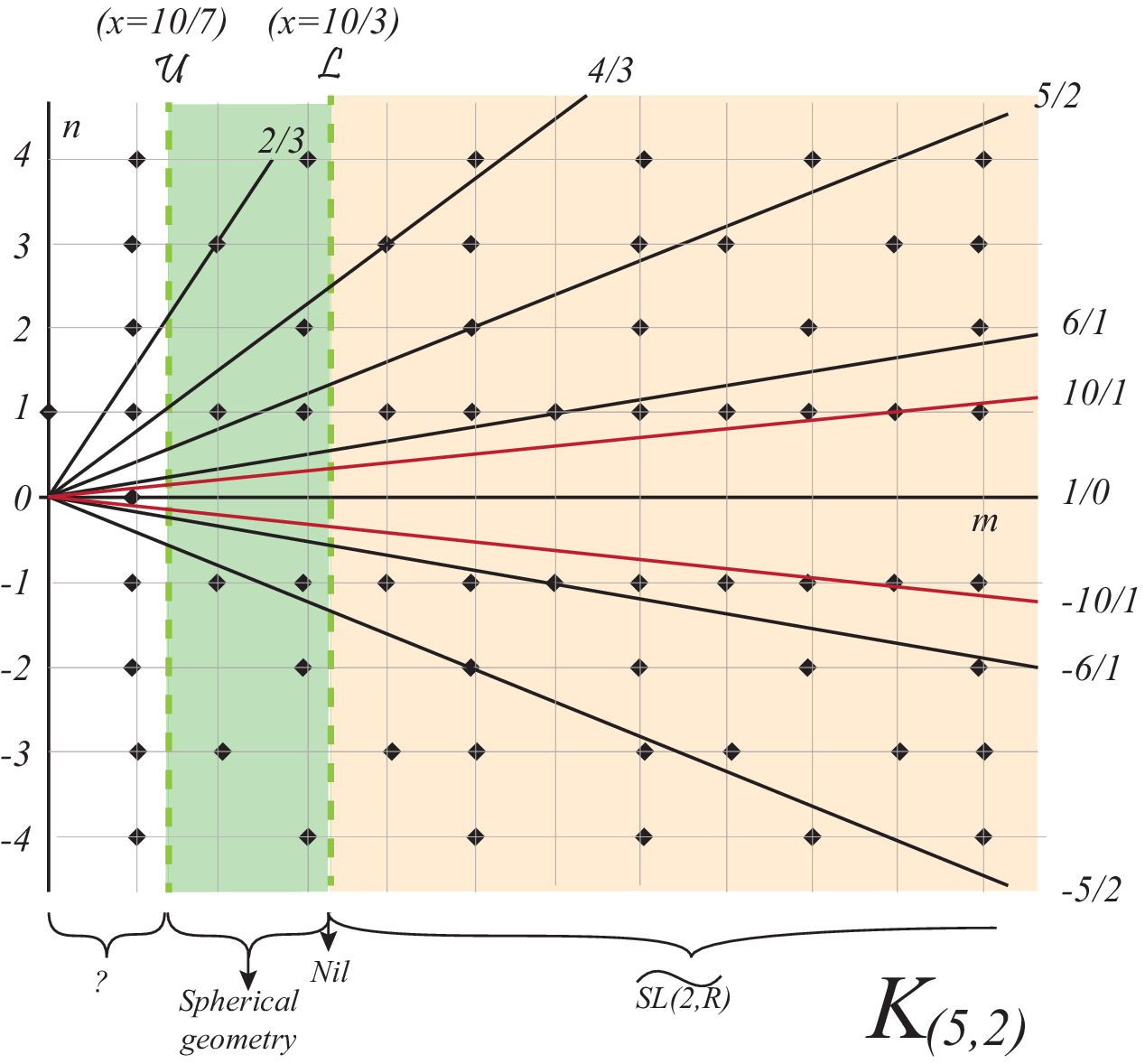,height=5.5cm}\epsfig{file=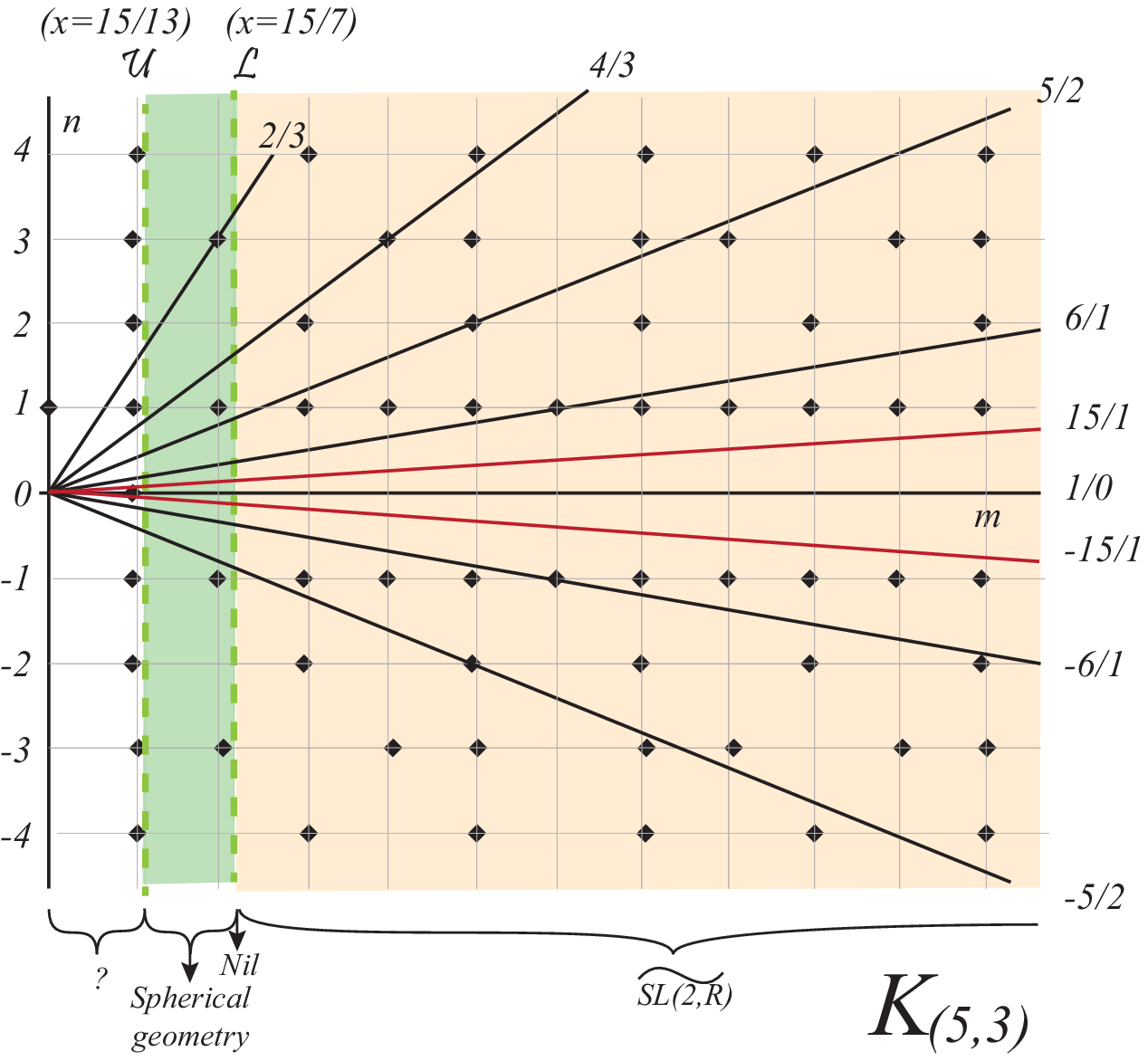,height=5.7cm}\\
\epsfig{file=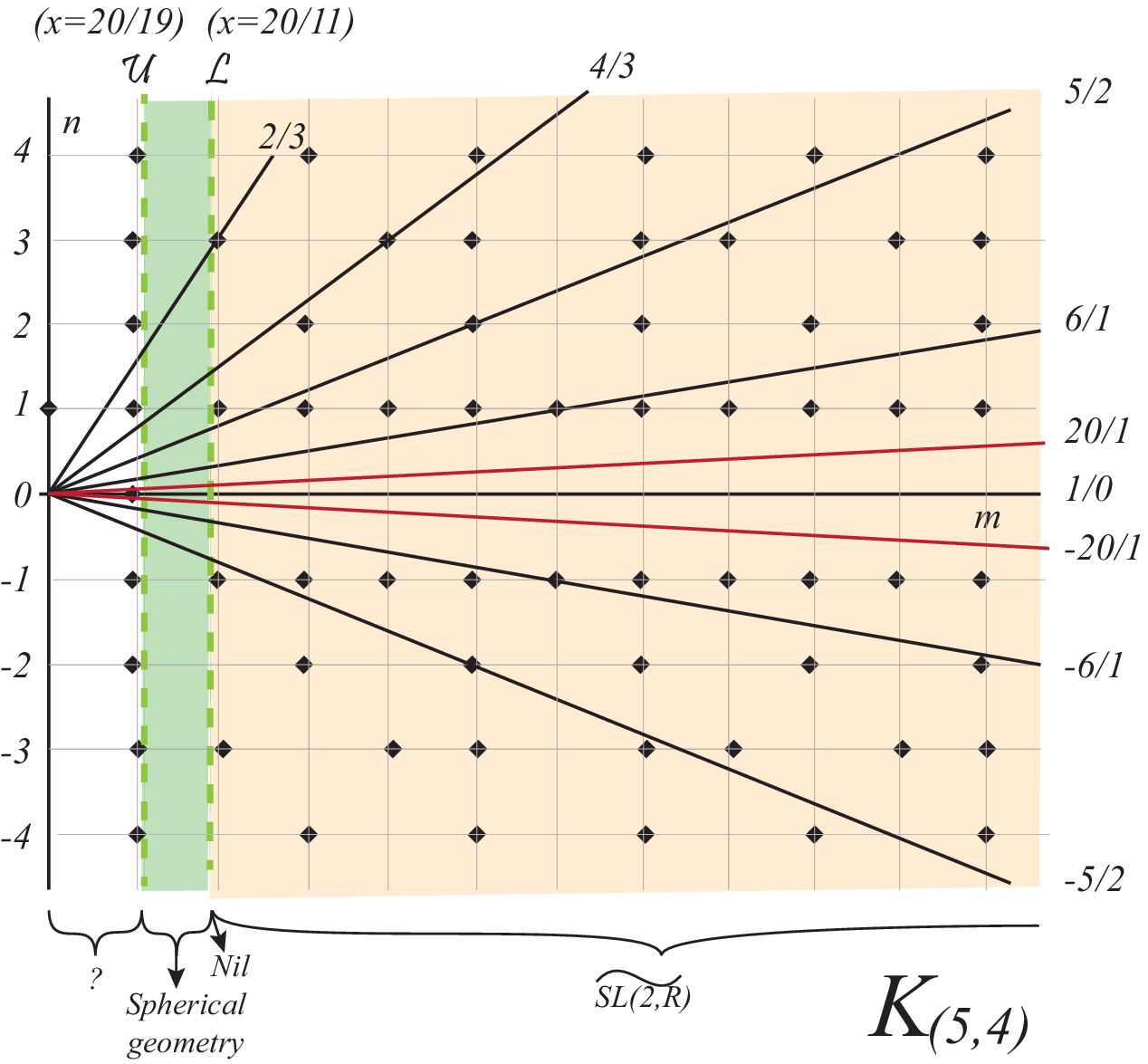,height=5.7cm}
\caption{The graph for knots $K_{(5,s)}$, ($s=2,3,4$)} \label{fplots}\end{center}
\end{figure}
\begin{enumerate}
\item Points in the line $l_{5s/1}$ ($l_{5s/(-1)}$) represent the geometric conemanifold structures in the manifold  $S_{(5,s)}\left( 0/1\right)$ ($S'_{(5,s)}\left( 0/1\right)$). The possible geometries in both cases are $(S^2\times\mathbb{R})$, Euclidean and $(H^2\times\mathbb{R})$. These conemanifolds structures are included in the next conclusions where we write in brackets the corresponding geometry.
\item For all $S_{(5,s)}\left( p/q\right)$ and $S'_{(5,s)}\left( p/q\right)$ the upper sphericity limit $\beta_{U}$ is equal to $7/3$ ($s=2$), $13/7$ ($s=3$), $19/11$ ($s=4$), times the lower sphericity limit $\beta_{L}$.
  \item There are no non-singular Nil (Euclidean) manifold structures on $S_{(5,s)}\left( p/q\right)$ and \newline $S'_{(5,s)}\left( p/q\right)$.
  \item There are no Nil (Euclidean) orbifold structures on $S_{(5,s)}\left( p/q\right)$ and $S'_{(5,s)}\left( p/q\right)$.
  \item There are no orbifold spherical $(S^2\times\mathbb{R})$ structures on $S_{(5,4)}\left( p/q\right)$ and $S'_{(5,4)}\left( p/q\right)$.
\item For $s=2,3$ and $v\in \mathbb{N}$, the orbifolds $\left(S_{(5,s)}\left( (|1-5sv|)/(-v)\right),2\right)$ and \newline $\left(S'_{(5,s)}\left( (1+5sv)/v\right),2\right)$ have a spherical orbifold structure with the core of the surgery as singular set labeled 2. Here
    \begin{eqnarray*}
      S_{(5,2)}\left( \frac{|1-10v|}{-v} \right) &=&\left\langle O,o,0\, |\, -1;(2,1),(5,3),(1,v)\right\rangle =\\ &=&(O,o,0\, |\, v-1;(2,1),(5,3))= L(10v+1,15v-1) \\
     S_{(5,3)}\left( \frac{|1-15v|}{-v} \right) &=& (O,o,0\, |\, v-1;(3,2),(5,2))=L(15v+1,9v)
     \end{eqnarray*}
     \begin{eqnarray*}
      S'_{(5,2)}\left( \frac{1+10v}{v} \right) &=&(O,o,0\, |\, v-1;(2,1),(5,2))=L(10v-1,15v-4) \\
      S'_{(5,3)}\left( \frac{1+15v}{v} \right) &=& (O,o,0\, |\, v-1;(3,1),(5,3))=L(15v-1,20v-3).
    \end{eqnarray*}
\item The orbifolds  $\left(S_{(5,2)}\left( (|1-10v|)/(-v)\right),3\right)$ and $\left(S'_{(5,2)}\left( (1+10v)/v\right),3\right)$, $v\in \mathbb{N}$,  have a spherical orbifold structure with the core of the surgery as singular set labeled 3.
\item The manifolds $S_{(5,s)}\left( (|2-5sv|)/(-v)\right)$ and $S'_{(5,s)}\left( (2+5sv)/v\right)$, $v$  odd integer, have a spherical manifold structure with the core of the surgery as geodesic. Here
     \begin{eqnarray*}
     S_{(5,2)}\left( \frac{10v-2}{-v}\right) &=& \left\langle O,o,0\, |\, -1;(2,1),(5,3),(2,v)\right\rangle =\\
     &=&(O,o,0\, |\, v-2;(2,1),(5,3),(2,1))\\
     S_{(5,3)}\left( \frac{15v-2}{-v}\right) &=& \left\langle O,o,0\, |\, -1;(3,2),(5,2),(2,v)\right\rangle =\\
     &=&(O,o,0\, |\, v-2;(3,2),(5,2),(2,1))
     \end{eqnarray*}
      \begin{eqnarray*}
     S_{(5,4)}\left( \frac{20v-2}{-v}\right) &=& \left\langle O,o,0\, |\, -1;(4,1),(5,4),(2,v)\right\rangle =\\
    &=&(O,o,0\, |\, v-2;(4,1),(5,4),(2,1)) \\
     S'_{(5,2)}\left( \frac{2+10v}{v}\right) &=& \left\langle O,o,0\, |\, -1;(2,1),(5,2),(2,v)\right\rangle =\\
    &=&(O,o,0\, |\, v-2;(2,1),(5,2),(2,1))
     \end{eqnarray*}
      \begin{eqnarray*}
     S'_{(5,3)}\left( \frac{2+15v}{v}\right) &=& \left\langle O,o,0\, |\, -1;(3,1),(5,3),(2,v)\right\rangle =\\
    &=&(O,o,0\, |\, v-2;(3,1),(5,3),(2,1)) \\
     S'_{(5,4)}\left( \frac{2+20v}{v}\right) &=& \left\langle O,o,0\, |\, -1;(4,3),(5,1),(2,v)\right\rangle =\\
     &=&(O,o,0\, |\, v-2;(4,3),(5,1),(2,1)).
    \end{eqnarray*}
\item The manifolds $S_{(5,3)}\left( (|3-15v|)/(-v)\right)$ and $S'_{(5,3)}\left( (3+15v)/v\right)$, where $\gcd (3,v)=1$, have a spherical manifold structure with the core of the surgery as geodesic. Here
   \begin{eqnarray*}
     S_{(5,3)}\left( (15v-3)/(-v)\right) &=& \left\langle O,o,0\, |\, -1;(3,2),(5,2),(3,v)\right\rangle \\
     S'_{(5,3)}\left( (3+15v)/v\right) &=& \left\langle O,o,0\, |\, -1;(3,1),(5,3),(3,v)\right\rangle \\
    \end{eqnarray*}
 \item All the other  manifold or orbifold structures in $S_{(5,s)}\left( p/q\right)$ and $S'_{(5,s)}\left( p/q\right)$ with  the core of the surgery as geodesic and possible singular set, have $\widetilde{SL(2,\mathbb{R})}$ $(H^2\times\mathbb{R})$ geometry.
\end{enumerate}

\subsection{Conclusions for the knots $K_{(r,s)}$ and $K^{*}_{(r,s)}$, $r>s>1$, $r>5$}\hspace*{\fill} \\

Figure \ref{fplotsrs} show the graph for the cases $r>5$.
\begin{figure}[h]
\begin{center}
\epsfig{file=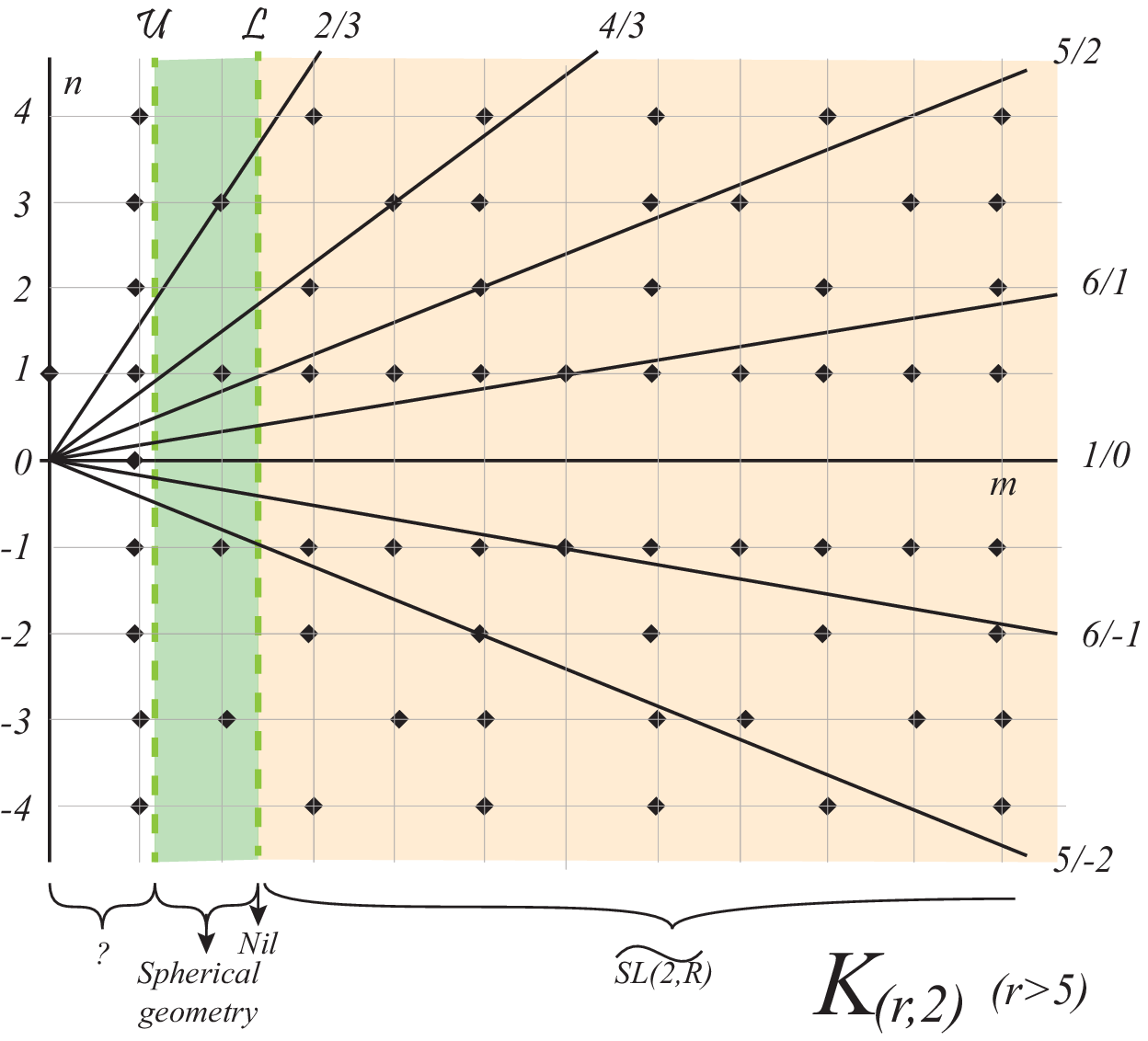,height=5.7cm}\epsfig{file=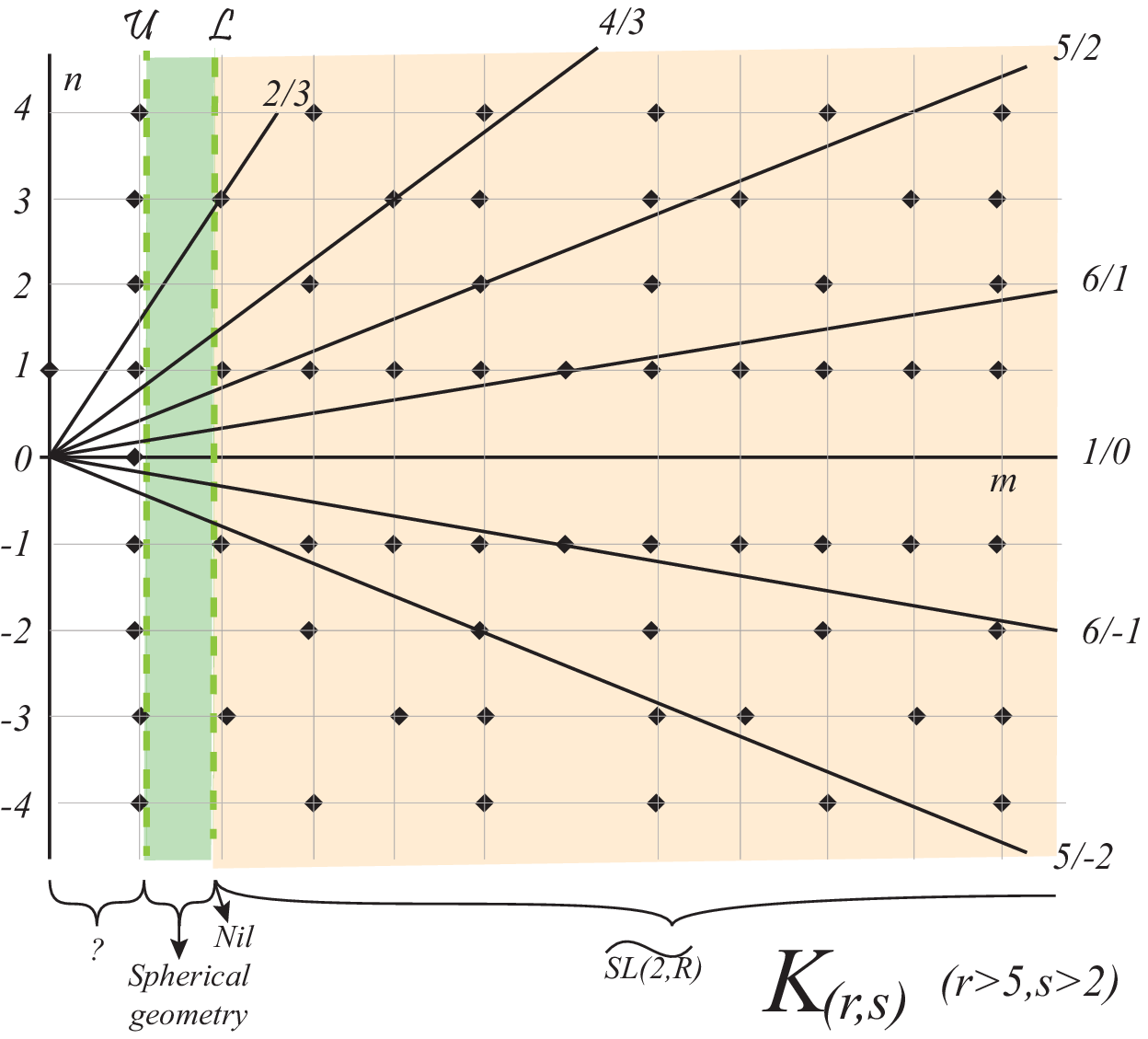,height=5.7cm}
\caption{The graph for knots $K_{(r,2)}$ and $K_{(r,s)}$, $r>5$, $s>2$} \label{fplotsrs}\end{center}
\end{figure}
\begin{enumerate}
\item Points in the line $l_{rs/1}$ ($l_{rs/(-1)}$) represent the geometric conemanifold structures in the manifold  $S_{(r,s)}\left( 0/1\right)$ ($S'_{(r,s)}\left( 0/1\right)$). The possible geometries in both cases are $(S^2\times\mathbb{R})$, Euclidean and $(H^2\times\mathbb{R})$. These conemanifolds structures are included in the next conclusions where we write in brackets the corresponding geometry.
  \item There are no non-singular Nil (Euclidean) manifold structures on $S_{(r,s)}\left( p/q\right)$ and \newline $S'_{(r,s)}\left( p/q\right)$.
  \item There are no Nil (Euclidean) orbifold structures on $S_{(r,s)}\left( p/q\right)$ and $S'_{(r,s)}\left( p/q\right)$.
  \item For $s>2$ there are not manifold or orbifold spherical $(S^2\times\mathbb{R})$ structures on $S_{(r,s)}\left( p/q\right)$ and $S'_{(r,s)}\left( p/q\right)$ with  the core of the surgery as geodesic or singular set.
\item The orbifolds $\left(S_{(r,2)}\left( (|1-2rv|)/(-v)\right),2\right)$ and $\left(S'_{(r,2)}\left( (1+2rv)/v\right),2\right)$, where $v\in \mathbb{N}$,  have a spherical orbifold structure with the core of the surgery as singular set labeled 2.
\item The manifolds $S_{(r,2)}\left( (|2-2rv|)/(-v)\right)$ and $S'_{(r,2)}\left( (2+2rv)/v\right)$, $v$  odd integer, have a spherical manifold structure with the core of the surgery as geodesic.
 \item All the other  manifold or orbifold structures in $S_{(r,s)}\left( p/q\right)$ and $S'_{(r,s)}\left( p/q\right)$ with  the core of the surgery as geodesic or singular set, have $\widetilde{SL(2,\mathbb{R})}$ $(H^2\times\mathbb{R})$ geometry.
\end{enumerate}

\subsection{Some examples}\hspace*{\fill} \\

\textbf{Surgery $\infty$. }For all cases, the horizontal line $y=0$ corresponds to surgery $\infty =1/0$ in a torus knot in the 3-sphere. Then the manifold is $S^{3}$ with the Seifert structure
 \[
 S^{3}= \left(O,o,0\, |\, -1;(s, b_1),(r, b_2)\right)
 \]
where $|-rs+b_1r+b_2s|=1$. This manifold has a spherical conemanifold structure with the knot $K_{(r,s)}$ as singular set with angle $\beta$
\[
2\frac{rs-r+s}{rs}\pi >\beta > 2\frac{rs-r-s}{rs}\pi ;
\]
Nil conemanifold structure, for
\[
\beta = 2\frac{rs-r-s}{rs}\pi ;
\]
 and  $\widetilde{SL(2,\mathbb{R})}$ conemanifold structure for
 \[
 \beta < 2\frac{rs-r-s}{rs}\pi .
 \]

(Actually $S^{3}$ has a manifold spherical structure where the knot $K_{(r,s)}$ is not a geodesic of that structure. However, in this paper, we are only studying geometric conemanifold (or manifold)  structures such that the fibres of the Seifert fibration are geodesic or singular.)

\textbf{Surgery 0.}

The values $p=0$, $q=1$, correspond to  $m=p\mp rsq=\mp rs$, $n=1$. The corresponding conemanifold is represented by the line with slope $\mp rs/1$. These Seifert manifolds have Euler class zero. The possible geometries are $(S^2\times\mathbb{R})$, Euclidean and $(H^2\times\mathbb{R})$.

For a torus knot $K_{(r,s)}$, $r>s>1$, $r>3$, the Seifert manifolds
\begin{eqnarray*}
  S_{(r,s)}\left( 0/1\right) &=&(O,o,0\, |\, -1;(s,b_1),(r,b_2),(rs,1))\\
  S'_{(r,s)}\left( 0/1\right)&=& (O,o,0\, |\, -1;(s,b_1),(r,b_2),(rs,-1))
\end{eqnarray*}
do not have $(S^2\times\mathbb{R})$ or Euclidean orbifold structure with the core of the surgery as singular set and all the orbifols structures have $(H^2\times\mathbb{R})$ geometry. The reason is that, in the graph, the line with slope $rs/\pm 1$ do not contain any point with integer coordinates in the spherical zone.

\bibliographystyle{plain}

\begin{thebibliography}{10}

\bibitem{D1988}
W.D. Dunbar
\newblock Geometric orbifolds.
\newblock {\em Revista Matem\'{a}tica de la UCM}, 1: 67--99,  1988.

\bibitem{FM1997}
 A. Fomenko, S.V. Matveev
 \newblock Algorithmic and computer methods for three-manifolds.
 \newblock {\em Mthematics an its Applications.} Vol 425, Kluwer Academic Publishers, Dordrecht, 1997.

\bibitem{HLM1992}
  H. Hilden, M. T. Lozano, J. M. Montesinos-Amilibia
 \newblock On the Borromean orbifolds: geometry and arithmetic.
 \newblock {\em Topology '90 (Columbus, OH, 1990), 133–167, Ohio State Univ. Math. Res. Inst. Publ., 1, de Gruyter, Berlin, 1992}

\bibitem{HLM1995}
  H. Hilden, M. T. Lozano, J. M. Montesinos-Amilibia
 \newblock On a remarkable polyhedron geometrizing the figure eight knot cone manifolds.
 \newblock {\em J. Math. Sci. Univ. Tokyo}, 2: 501--561, 1995

\bibitem{LM2014}
M. T. Lozano, J. M. Montesinos-Amilibia
\newblock On the degeneration of some 3-manifolds geometries via unit groups of quaternion algebras.
\newblock {\em RACSAM } 109: 669--715, 2015.

\bibitem{LM2015}
M. T. Lozano, J. M. Montesinos-Amilibia
\newblock Geometric conemanifolds structures on $\mathbb{T}_{p/q}$, the result of $p/q$ surgery in the left-handed trefoil knot $\mathbb{T}$.
\newblock {\em : J. Knot Theory  Ramifications} Vol. 24, No. 12 (2015) 1550057 (38 pages)
DOI: 10.1142/S0218216515500571

\bibitem{M1987}
J.M. Montesinos-Amilibia
\newblock Classical Tessellations and Three Manifolds
\newblock {Springer-Verlag } 1987.

\bibitem{M1971}
L. Moser
\newblock Elementary surgery along a torus knot.
\newblock {\em Pacific Journal of Mathematics}, 38: 737--745,  1971.

\bibitem{O1972}
P. Orlik
\newblock Seifert manifolds.
\newblock {\em Lecture Notes in Mathematics}, 291 Springer-Verlag  1972.

\bibitem{S1983}
P. Scott
\newblock The geometries of 3-manifolds.
\newblock {\em Bull. London Math. Soc.}, 15: 401--487,  1983.

\bibitem{S1933}
H. Seifert
\newblock Topologie dreidimensionaler gefaserter R\"{a}ume.
\newblock{\em Acta Math.} 60:147-238 ,  1933.

 \bibitem{T1980}
W. P. Thurston
\newblock The Geometry and Topology of Three-Manifolds.
\newblock {\em Princeton University  1980}


\end{thebibliography}

\end{document}